\documentclass[journal,twoside,web]{ieeecolor}

\makeatletter
\let\NAT@parse\undefined
\makeatother

\usepackage{generic}
\usepackage{amsfonts,amsopn,amsmath,amssymb}
\makeatletter
\@ifclassloaded{ieeecolor}{
  \usepackage{etoolbox}
  \providecommand{\refname}{References}
  \patchcmd{\thebibliography}{\section*{\color{black}}}{\section*{\refname}}{}{}
  \let\oldmaketitle\maketitle
  \def\maketitle{%
    \let\oldcenterline\centerline
    \def\centerline##1{\par\oldcenterline{##1}}%
    \oldmaketitle
  }
}{
  \usepackage{xcolor}
}
\makeatother
\usepackage{supertabular,booktabs,array}
\usepackage{flushend}
\usepackage{graphicx}
\usepackage{float}
\usepackage[mathscr]{eucal}
\usepackage{hyperref,cleveref}

\newtheorem{theorem}{Theorem}
\newtheorem{lemma}{Lemma}
\newtheorem{corollary}{Corollary}
\newtheorem{definition}{Definition}
\newtheorem{example}{Example}

\def\BibTeX{{\rm B\kern-.05em{\sc i\kern-.025em b}\kern-.08em
    T\kern-.1667em\lower.7ex\hbox{E}\kern-.125emX}}
\begin{document}

\title{Dual-Regularized Riccati Recursions for Interior-Point Optimal Control}

\author{
  \href{mailto:joaospinto@gmail.com}{João Sousa-Pinto}
  \and \qquad
  \href{mailto:dominique.orban@polymtl.ca}{Dominique Orban}
}
\makeatother

\ifpdf
\hypersetup{
  pdftitle={Dual-Regularized Riccati Recursions for Interior-Point Optimal Control},
  pdfauthor={J. Sousa-Pinto, D. Orban}
}
\fi

\maketitle

\begin{abstract}
We derive closed-form extensions of the sequential~\cite{lqr} and parallel~\cite{lqr-gpu}
Riccati recursions for solving \emph{dual-regularized} linear-quadratic regulator (LQR) problems,
with $O(N)$ sequential time and $O(\log(N))$ parallel time, respectively.
We show that these subproblems arise when using regularized primal-dual interior-point methods
to solve smooth, constrained, non-convex, discrete-time optimal control problems via multiple-shooting,
even in the presence of stagewise equality or inequality constraints, and
without imposing any rank requirements on constraint Jacobians.
We prove that, when certain inertia conditions on the Newton-KKT matrix are met,
each nonzero primal step is a descent direction of an
augmented barrier-Lagrangian merit function.
We characterize these inertia conditions in terms of the positive-definiteness of the
dual-regularized Riccati pivots (a weaker condition than the standard LQR positive-definiteness requirements),
thereby yielding inexpensive certificates of the required inertia.
We provide MIT-licensed implementations of our methods in C++~\cite{sip,sipoc} and in JAX~\cite{lipa,rlqr-jax},
as well as a full formalization of our results in Lean~\cite{rlqr-lean}.
We benchmark our algorithm against leading optimal control and nonlinear programming solvers
on complex trajectory optimization problems, establishing competitive performance on moderate problems and
substantial gains as the horizon length, problem dimension, and constraint count increase.
\end{abstract}

\begin{keywords}
Optimal Control, Numerical Optimization, Interior Point Method
\end{keywords}


\section{Introduction}\label{introduction}

\subsection{Numerical Optimal Control}\label{numerical-optimal-control-intro}

Optimal control theory studies optimization problems in which
the cost functional being optimized depends on the state ($x$) and control ($u$) time series
of a dynamical system, whose dynamics are described by an ordinary differential
equation $\dot{x}(t) = f(x(t), u(t), t)$ (in the case of continuous-time problems)
or by a difference equation $x_{n+1} = f_n(x_n, u_n)$
(in the case of discrete-time problems).

Numerical optimal control, both real-time and offline, has found numerous applications,
ranging from trajectory optimization for robotics (e.g. for autonomous cars,
unmanned aerial vehicles, legged robots) and aerospace (e.g. rockets, satellites)
to the control of power systems (e.g. power plants, electrical grid).

Continuous-time optimal control problems, whose optimization variables are functions
(and thus infinite-dimensional), are typically converted into finite-dimensional optimization problems
by either shooting (i.e. replacing $x(t), u(t)$ with sampled values $(x_i)_{i=0}^N, (u_i)_{i=0}^{N-1}$)
or collocation (i.e. expressing $x(t), u(t)$ as finitely parameterized functions such as splines).
This process is called transcription.
When the problem is transcribed via a shooting method and an explicit integration scheme
is used to discretize the dynamics, the resulting optimization problem is a discrete-time
optimal control problem.

Depending on the application domain, different transcription mechanisms may be more suitable.
In robotics, where optimal control problems are often solved at a high frequency,
schemes resulting in faster optimization times are often favored over those that may
yield more precise solutions at the cost of slower optimization times. The opposite may be true
of other application domains, such as offline trajectory optimization for satellites or rockets.

Henceforth, we shall only discuss discrete-time optimal control problems. Solution mechanisms
for these problems can be roughly divided into two groups: single-shooting and multiple-shooting
methods. In the case of single-shooting, only the control variables $u_i$ are free optimization
variables, and the state variables $x_i$ are treated as dependent variables; in particular, the
dynamics of the system are satisfied at every iterate. However, this comes at the cost of making
it substantially harder to warm-start the optimization process.
The most common single-shooting methods are the Differential Dynamic Programming (DDP)
algorithm~\cite{ddp} and the Stagewise Newton algorithm~\cite{snewton}.
In multiple-shooting methods, some or all of the state variables $x_{n+1}$
are also treated as independent variables, and the dynamics equations
$x_{n+1} = f_n(x_n, u_n)$ for such $n$ are enforced as general constraints.
This formulation has the advantage of being easy to warm-start and easily
paired with generic optimizers of nonlinear programs.
However, using optimal-control-specific subproblem solvers is crucial for performance.

When the dynamics of the system are affine, the costs are stagewise convex quadratics,
strictly convex with respect to the control variables, and no further constraints are present,
unconstrained discrete-time optimal control problems can be solved exactly
in a single backward and forward pass, as first shown in~\cite{lqr}.
Such problems are called linear-quadratic regulator (LQR) problems.
The backward pass computes an optimal affine control policy $u_i = K_i x_i + k_i$
in $O(N)$ time, and the forward pass alternates between computing the next state
by applying the dynamics law and computing the next control by applying the control law.
The closed-form equations that produce this backward pass are called the sequential Riccati recursion.
This remains true in the presence of stagewise affine constraints, as first shown in~\cite{lclqr}.

The parallel Riccati recursion, first derived in~\cite{lqr-gpu}, replaces the backward and forward passes
with reverse and forward associative scans, respectively, enabling the solution of LQR problems in
$O(\log(N))$ parallel time. This requires reframing the derivation in terms of interval value functions,
rather than cost-to-go functions.

The canonical method for applying an interior-point method (IPM) to constrained discrete-time
optimal control problems via multiple shooting, first described in~\cite{mpc-ipm}, consists of exploiting
the block sparsity of the resulting Newton-KKT systems by performing block elimination of all slack variables
and all dual variables associated with inequality constraints.
This transformation converts the Newton-KKT systems of the constrained problems into first-order optimality
conditions associated with LQR problems (with affine stagewise constraints,
when the full problem has stagewise equality constraints). Then, a Riccati recursion
can be used to solve the post-elimination linear systems, and the remaining variables can be recovered from the
block-elimination equations.

\subsection{Motivation}\label{motivation}

Regularized and proximal variants of the primal-dual interior-point method are commonly
used to improve robustness on nonconvex and degenerate constrained optimization
problems. By modifying the Newton-KKT systems with primal and/or dual
regularization, these methods can enforce the inertia conditions needed for
merit function descent even when the unregularized systems are singular,
ill-conditioned, or fail to have the desired inertia.

In the regularized and proximal versions of the interior-point method, dual-variable regularization $-\frac{1}{\eta}$
is added to the diagonal of the dual-variable blocks of the left-hand side of the Newton-KKT linear systems posed at each iteration.

In the regularized IPM case, the primal component of the Newton-KKT solution is guaranteed to descend an
augmented barrier-Lagrangian merit function with penalty parameter $\eta$.
In the proximal IPM case, the full primal-dual Newton-KKT solution is guaranteed to descend
an augmented barrier-Lagrangian merit function with the dual-regularization terms introduced in~\cite{pdal},
also with penalty parameter $\eta$.

In either case, the block-elimination procedure from~\cite{mpc-ipm} produces first-order optimality
conditions associated with \textit{dual-regularized} LQR problems, which are the key object of this paper.
When stagewise equality constraints are present, an extra step for block-eliminating the
corresponding dual variables is required; this step becomes possible due to the presence of the
diagonal $-\frac{1}{\eta}$ terms.

\subsection{Contributions}\label{contributions}

The main contribution of this paper is a set of efficient, closed-form solutions for
dual-regularized LQR problems, both sequential (time complexity $O(N)$)
and parallel (parallel time $O(\log(N))$).

We also provide a substantial speedup over the parallel algorithm introduced in~\cite{lqr-gpu}
to solve standard (i.e. not dual-regularized) LQR problems. Specifically, we show that each
interval value function combination step can be done with a single matrix factorization.

We prove that the inertia condition needed for descent of the regularized
IPM merit function can be certified exactly from the Riccati factorization:
the full Newton-KKT system has the desired inertia if and only if the
dual-regularized Riccati pivots are positive-definite.

Finally, we release MIT-licensed implementations of all the methods described herein,
both in C++~\cite{sip,sipoc} and in JAX~\cite{lipa,rlqr-jax},
paired with benchmarks that compare them against state-of-the-art
packages for solving numerical optimal control problems.

Formal proofs (in Lean) of all the results from this paper are available in~\cite{rlqr-lean}.

\subsection{Related Work}\label{related-work}

The primal-dual constrained DDP~\cite{pd-ddp} and ProxDDP~\cite{prox-ddp} methods
are the closest to our work, as the Newton-KKT systems they pose can be interpreted as
first-order optimality conditions of dual-regularized LQR problems,
even though they are not interior-point methods.
Unlike them, we do not solve the Newton-KKT subproblems by
factoring the full per-stage components of the primal-dual KKT matrices.
Instead, we block-eliminate all dual variables (except for the costates
and the initial state constraint Lagrange multiplier), and reduce the subproblem to a
dual-regularized LQR problem, to which we then provide efficient closed-form solutions.
As a result, the per-stage matrices factored by our backward pass are substantially smaller:
they depend only on the state and control dimensions, and not on the number of constraints,
and the same algebra admits an efficient parallel solution scheme.

FATROP~\cite{fatrop} uses an IPOPT~\cite{ipopt}-inspired primal-dual interior-point method with Hessian regularization
and a structure-exploiting Riccati-type linear solver capable of handling stagewise equality constraints,
but imposes full row-rank requirements on the constraint Jacobians.
In contrast, we use a regularized interior-point method, and the dual-regularization of our Newton-KKT systems
allows them to be reduced to dual-regularized LQR problems (even in the presence of stagewise equality constraints),
without imposing rank requirements on the constraint Jacobians.

The regularized IPM used in this paper is closely related to the mixed logarithmic
barrier-augmented Lagrangian methods studied for general nonlinear programs
in~\cite{armand-omheni2017mixed,armand-tran2019rapid,nguyen2025convergence},
where strong global convergence properties and local superlinear convergence are established,
to both feasible and infeasible stationary points.

An efficient and exact parallel algorithm for standard (i.e. not dual-regularized) LQR problems
was first established in~\cite{lqr-gpu}.

\section{Background}\label{background}

\subsection{Dual-Regularized LQR}\label{dual-regularized-lqr}

\begin{definition}
A dual-regularized LQR problem is an optimization problem of the form

\begin{equation}
\begin{aligned}
\label{min-max-with-ys}
\max \limits_{y_0, \ldots, y_N} \min \limits_{x_0, u_0, \ldots, x_N}
\mathcal{R}(x_0, u_0, \ldots, x_N, y_0, \ldots, y_N),
\end{aligned}
\end{equation}
where
\begin{equation*}
\begin{aligned}
& \mathcal{R}(x_0, u_0, \ldots, x_N, y_0, \ldots, y_N) = \\
& \sum \limits_{i=0}^{N-1}
\left( \frac{1}{2} x_i^T Q_i x_i + x_i^T M_i u_i + \frac{1}{2} u_i^T R_i u_i
+ q_i^T x_i + r_i^T u_i \right) \\
& + \frac{1}{2} x_N^T Q_N x_N + q_N^T x_N
- \sum \limits_{i=0}^{N} \frac{1}{2} y_i^T \Delta_i y_i
+ y_0^T (c_0 - x_0) \\
& + \sum \limits_{i=0}^{N-1} \left( y_{i+1}^T (A_i x_i + B_i u_i + c_{i+1} - x_{i + 1}) \right) .
\end{aligned}
\end{equation*}

The matrices
\begin{equation*}
\begin{aligned}
\begin{bmatrix}
Q_i & M_i \\
M_i^T & R_i
\end{bmatrix}
\end{aligned}
\end{equation*}
are assumed to be symmetric and positive semi-definite,
and the matrices $R_i$ and $\Delta_i$ are assumed to be symmetric and positive-definite.
Typically, the matrices $\Delta_i$ would also be diagonal,
but we do not impose this requirement.
\end{definition}

When $\Delta = 0$, these problems become standard (i.e. not dual-regularized) LQR problems.

The first-order optimality conditions of a dual-regularized LQR problem are a linear system of the form
\begin{equation}
\begin{aligned}
\label{primal-dual-linear-system}
\begin{bmatrix}
P & C^T \\
C & -\Delta
\end{bmatrix}
\begin{bmatrix}
x \\
y
\end{bmatrix} = -
\begin{bmatrix}
s \\
c
\end{bmatrix},
\end{aligned}
\end{equation}
where
\begin{equation}
\begin{aligned}
\label{optimal-control-specialization}
P &=
\begin{bmatrix}
P_0 & & \\
& \ddots & \\
& & P_N
\end{bmatrix}, \quad
\Delta = \begin{bmatrix}
\Delta_0 & & \\
& \ddots & \\
& & \Delta_N
\end{bmatrix} , \\
P_i &=
\begin{cases}
\begin{bmatrix}
Q_i & M_i \\
M_i^T & R_i
\end{bmatrix}, & \text{if } 0 \leq i < N, \\
Q_i, & \text{if } i = N,
\end{cases} \\
C &=
\begin{bmatrix}
-I  &     &     &        &         &         & \\
A_0 & B_0 &  -I &        &         &         & \\
    &     &     & \ddots &      -I &         & \\
    &     &     &        & A_{N-1} & B_{N-1} & -I
\end{bmatrix}, \\
s &=
\begin{bmatrix}
q_0 \\
r_0 \\
\vdots \\
q_{N-1} \\
r_{N-1} \\
q_N
\end{bmatrix},
c =
\begin{bmatrix}
c_0 \\
c_1 \\
\vdots \\
c_N
\end{bmatrix},
x =
\begin{bmatrix}
x_0 \\
u_0 \\
\vdots \\
x_{N - 1} \\
u_{N - 1} \\
x_N
\end{bmatrix},
y =
\begin{bmatrix}
y_0 \\
y_1 \\
\vdots \\
y_N
\end{bmatrix} .
\end{aligned}
\end{equation}

Note that~\cref{primal-dual-linear-system} is well-defined even when the positive-definiteness
assumptions mentioned above are not met. In fact, we will want to solve~\cref{primal-dual-linear-system}
even in this case. This will be discussed further in~\cref{riccati-rational-identity-theorem}.

\subsection{Optimal Control}\label{optimal-control}

\begin{definition}
A discrete-time optimal control problem is an optimization problem of the form

\begin{equation*}
\begin{aligned}
&\min_{x_0, u_0, \ldots, x_N} &\sum \limits_{i=0}^{N-1} f_i(x_i, u_i) + f_N(x_N) \\
&\mbox{s.t.} & x_0 = s_0, \\
& & \forall i \in \lbrace 0, \ldots, N-1 \rbrace, x_{i+1} = d_i(x_i, u_i),  \\
& & \forall i \in \lbrace 0, \ldots, N-1 \rbrace, c_i(x_i, u_i) = 0, \\
& & \forall i \in \lbrace 0, \ldots, N-1 \rbrace, g_i(x_i, u_i) \leq 0, \\
& & c_N(x_N) = 0, \\
& & g_N(x_N) \leq 0 .
\end{aligned}
\end{equation*}

The variables $x_i, u_i$ represent the states and controls;
the functions $d_i, f_i, c_i, g_i$ are the dynamics, costs, equality constraints,
and inequality constraints (respectively);
$s_0$ is the fixed initial state; $N$ is the number of stages.
\end{definition}

\subsection{The Regularized Interior Point Method}\label{interior-point}

The regularized interior-point method (IPM) solves optimization problems of the form

$$\min\limits_{x} f(x) \qquad \mbox{s.t.}
  \quad c(x) = 0 \wedge g(x) + s = 0 \wedge s \geq 0,$$

where the functions $f, c, g$ are continuously differentiable.

\begin{definition}
The Barrier-Lagrangian is defined as
\begin{equation*}
\begin{aligned}
\mathcal{L}(x, s, y, z; \mu) = & f(x) - \mu \sum \limits_{i} \log(s_i) \\
& + y^T c(x) + z^T (g(x) + s)
\end{aligned}
\end{equation*}
\end{definition}

\begin{definition}
If $M$ is symmetric and positive-definite, the squared $M$-norm is defined as
$\lVert x \rVert_{M}^2 = x^T M x$.
\end{definition}

\begin{definition}
The augmented barrier-Lagrangian is defined as
\begin{equation*}
\begin{aligned}
& \mathcal{A}(x, s, y, z; \mu, \Delta_C, \Delta_G) \\
= & \mathcal{L}(x, s, y, z; \mu) + \frac{1}{2} (\lVert c(x) \rVert_{\Delta_C}^2 + \lVert g(x) + s\rVert_{\Delta_G}^2),
\end{aligned}
\end{equation*}
where $\Delta_C$ and $\Delta_G$ are symmetric positive-definite matrices.
\end{definition}

Frequently, $\Delta_C$ and $\Delta_G$ are set to $\eta I$, where $\eta > 0$.
However, having separate penalty parameters for different constraints
may promote numerical stability and faster convergence.

We call $x, s$ the primal variables and $y, z$ the dual variables.
As usual, we require that $s, z > 0$ always hold.

At each iteration of the regularized interior-point method, the search direction
$(\Delta x, \Delta s, \Delta y, \Delta z)$ is computed by solving the linear system

\begin{equation}
\begin{aligned}
\label{ipm-4x4-newton-kkt}
\begin{bmatrix}
P & 0 & C^T & G^T \\
0 & W^{-1} & 0 & I \\
C & 0 & -\Delta_C^{-1} & 0 \\
G & I & 0 & -\Delta_G^{-1}
\end{bmatrix}
\begin{bmatrix}
\Delta x \\
\Delta s \\
\Delta y \\
\Delta z
\end{bmatrix}
= - \nabla \mathcal{L},
\end{aligned}
\end{equation}
where
$\nabla \mathcal{L}$ is evaluated at $(x, s, y, z; \mu)$,
$P = \nabla^2_{xx} \mathcal{L}(x, s, y, z; \mu)$ (or any symmetric approximation thereof),
$C = J(c)(x)$, $G = J(g)(x)$,
$S, Z$ denote the diagonal matrices with entries $s, z$ (respectively),
and $W = Z^{-1} S$ (or any symmetric positive-definite approximation thereof).

The variable $\Delta s$ can be eliminated from~\cref{ipm-4x4-newton-kkt} via
\begin{equation*}
\begin{aligned}
\Delta s &= -W (\Delta z - \mu S^{-1} e + z),
\end{aligned}
\end{equation*}
where $e$ denotes the all-ones vector, resulting in the linear system

\begin{equation}
\begin{aligned}
\label{ipm-3x3-newton-kkt}
& \begin{bmatrix}
P & C^T & G^T \\
C & -\Delta_C^{-1} & 0 \\
G & 0 & -(W + \Delta_G^{-1})
\end{bmatrix}
\begin{bmatrix}
\Delta x \\
\Delta y \\
\Delta z
\end{bmatrix}
\\ = & -
\begin{bmatrix}
\nabla_x \mathcal{L}(x, s, y, z; \mu) \\
c(x) \\
g(x) + s + W (\mu S^{-1} e - z)
\end{bmatrix}.
\end{aligned}
\end{equation}

We now show that the primal search direction determined by the regularized
interior-point method is guaranteed to be a descent direction of
$\mathcal{A}(\cdot, \cdot, y, z; \mu, \Delta_C, \Delta_G)$
at $(x, s)$ unless the primal variables have already converged,
provided some conditions on the inertia of the Newton-KKT matrix from~\cref{ipm-4x4-newton-kkt} are met.

\begin{definition}
For a symmetric matrix $M$, we denote its inertia by
$\operatorname{In}(M) = (n_+(M), n_-(M), n_0(M))$,
where $n_+(M)$, $n_-(M)$, and $n_0(M)$ denote the number of positive,
negative, and zero eigenvalues of $M$, counted with multiplicity.
\end{definition}

\begin{lemma}
\label{sylvester-inertia-lemma}
Let $H$ and $D$ be symmetric matrices with $D$ nonsingular. Then
\begin{equation*}
\begin{aligned}
\operatorname{In}\left(
\begin{bmatrix}
H & A^T \\
A & -D
\end{bmatrix}
\right)
=
\operatorname{In}(-D) + \operatorname{In}(H + A^T D^{-1} A).
\end{aligned}
\end{equation*}
\end{lemma}

\begin{proof}
Let $\Sigma = H + A^T D^{-1} A$ and
$T =
\begin{bmatrix}
I & 0 \\
-D^{-1} A & I
\end{bmatrix}$.
Then
\begin{equation*}
\begin{aligned}
\begin{bmatrix}
H & A^T \\
A & -D
\end{bmatrix}
= T^T
\begin{bmatrix}
\Sigma & 0 \\
0 & -D
\end{bmatrix}
T.
\end{aligned}
\end{equation*}
Since $T$ is nonsingular, Sylvester's law of inertia implies that the
left-hand side and the block-diagonal matrix on the right-hand side have the
same inertia. The result follows by additivity of inertia over block-diagonal
matrices.
\end{proof}

For compactness, we let $K_4$ denote the matrix in~\cref{ipm-4x4-newton-kkt}
and $K_3$ denote the matrix in~\cref{ipm-3x3-newton-kkt}.
We also let $n_x$ denote the dimension of $x$, and let $n_c$ and $n_g$
denote the number of equality and inequality constraints, respectively.

\begin{lemma}
\label{4x4-3x3-inertia-lemma}
\begin{equation*}
\begin{aligned}
\operatorname{In}(K_4) = \operatorname{In}(K_3) + (n_g,0,0).
\end{aligned}
\end{equation*}
\end{lemma}

\begin{proof}
Permuting the variables of $K_4$ from $(x,s,y,z)$ to $(x,y,z,s)$ yields
\begin{equation*}
\begin{aligned}
\begin{bmatrix}
P & C^T & G^T & 0 \\
C & -\Delta_C^{-1} & 0 & 0 \\
G & 0 & -\Delta_G^{-1} & I \\
0 & 0 & I & W^{-1}
\end{bmatrix}.
\end{aligned}
\end{equation*}
Applying~\cref{sylvester-inertia-lemma} with $D=-W^{-1}$ gives the Schur
complement
\begin{equation*}
\begin{aligned}
\begin{bmatrix}
P & C^T & G^T \\
C & -\Delta_C^{-1} & 0 \\
G & 0 & -(W+\Delta_G^{-1})
\end{bmatrix}
= K_3 .
\end{aligned}
\end{equation*}
Since $W^{-1}$ is positive-definite, the eliminated block contributes
$(n_g,0,0)$ to the inertia.
\end{proof}

\begin{lemma}
\label{3x3-2x2-inertia-lemma}
Let
\begin{equation*}
\begin{aligned}
K_{xy} =
\begin{bmatrix}
P + G^T (W + \Delta_G^{-1})^{-1} G & C^T \\
C & -\Delta_C^{-1}
\end{bmatrix},
\end{aligned}
\end{equation*}
Then
\begin{equation*}
\begin{aligned}
\operatorname{In}(K_3) = \operatorname{In}(K_{xy}) + (0,n_g,0),
\end{aligned}
\end{equation*}
or, equivalently,
\begin{equation*}
\begin{aligned}
\operatorname{In}(K_4) = \operatorname{In}(K_{xy}) + (n_g,n_g,0).
\end{aligned}
\end{equation*}
\end{lemma}

\begin{proof}
Applying~\cref{sylvester-inertia-lemma} to $K_3$ with
$D=W+\Delta_G^{-1}$ eliminates $\Delta z$ and gives the Schur complement
$K_{xy}$. Since $D$ is positive-definite, the eliminated block $-D$
contributes $(0,n_g,0)$ to the inertia. The equivalent statement follows
from~\cref{4x4-3x3-inertia-lemma}.
\end{proof}

\begin{lemma}
\label{4x4-primal-inertia-lemma}
Let
\begin{equation*}
\begin{aligned}
K_{xs} =
\begin{bmatrix}
P & 0 \\
0 & W^{-1}
\end{bmatrix}
+
\begin{bmatrix}
C^T & G^T \\
0 & I
\end{bmatrix}
\begin{bmatrix}
\Delta_C & 0 \\
0 & \Delta_G
\end{bmatrix}
\begin{bmatrix}
C & 0 \\
G & I
\end{bmatrix},
\end{aligned}
\end{equation*}
Then
\begin{equation*}
\begin{aligned}
\operatorname{In}(K_4) = \operatorname{In}(K_{xs}) + (0,n_c+n_g,0).
\end{aligned}
\end{equation*}
\end{lemma}

\begin{proof}
Let
\begin{equation*}
\begin{aligned}
H =
\begin{bmatrix}
P & 0 \\
0 & W^{-1}
\end{bmatrix},
\quad
A =
\begin{bmatrix}
C & 0 \\
G & I
\end{bmatrix},
\quad
D =
\begin{bmatrix}
\Delta_C^{-1} & 0 \\
0 & \Delta_G^{-1}
\end{bmatrix}.
\end{aligned}
\end{equation*}
Then $K_4 =
\begin{bmatrix}
H & A^T \\
A & -D
\end{bmatrix}$
and $K_{xs}=H+A^TD^{-1}A$. Applying~\cref{sylvester-inertia-lemma}
shows that
\begin{equation*}
\begin{aligned}
\operatorname{In}(K_4)=\operatorname{In}(-D)+\operatorname{In}(K_{xs}).
\end{aligned}
\end{equation*}
Since $D$ is positive-definite, $\operatorname{In}(-D)=(0,n_c+n_g,0)$.
\end{proof}

\begin{lemma}
\label{shifted-kkt-al-gradient-lemma}
The linear system
\begin{equation}
\begin{aligned}
& \begin{bmatrix}
P & 0 & C^T & G^T \\
0 & W^{-1} & 0 & I \\
C & 0 & -\Delta_C^{-1} & 0 \\
G & I & 0 & -\Delta_G^{-1}
\end{bmatrix}
\begin{bmatrix}
\Delta x \\
\Delta s \\
\Delta y - \Delta_C c(x) \\
\Delta z - \Delta_G (g(x) + s)
\end{bmatrix}
\\ = & -
\begin{bmatrix}
\nabla_x \mathcal{A}(x, s, y, z; \mu, \Delta_C, \Delta_G) \\
\nabla_s \mathcal{A}(x, s, y, z; \mu, \Delta_C, \Delta_G) \\
0 \\
0
\end{bmatrix}
\end{aligned}
\end{equation}
is equivalent to~\cref{ipm-4x4-newton-kkt}.
\end{lemma}

Below, we use $D( \cdot ; \cdot )$ to represent the directional derivative operator.

\begin{lemma}
\label{al-directional-derivative-lemma}
The primal component $(\Delta x,\Delta s)$ of the solution of
\cref{ipm-4x4-newton-kkt} satisfies
\begin{equation*}
\begin{aligned}
&D(\mathcal{A}(\cdot, \cdot, y, z; \mu, \Delta_C, \Delta_G);
(\Delta x, \Delta s))(x, s) \\
=& -
\begin{bmatrix}
\Delta x \\
\Delta s
\end{bmatrix}^T
K_{xs}
\begin{bmatrix}
\Delta x \\
\Delta s
\end{bmatrix}.
\end{aligned}
\end{equation*}
\end{lemma}

\begin{proof}
\begin{equation*}
\begin{aligned}
&D(\mathcal{A}(\cdot, \cdot, y, z; \mu, \Delta_C, \Delta_G);
(\Delta x, \Delta s))(x, s) \\
=&
\begin{bmatrix}
\Delta x^T & \Delta s^T & 0 & 0
\end{bmatrix}
\begin{bmatrix}
\nabla_x \mathcal{A}(x, s, y, z; \mu, \Delta_C, \Delta_G) \\
\nabla_s \mathcal{A}(x, s, y, z; \mu, \Delta_C, \Delta_G) \\
0 \\
0
\end{bmatrix} \\
=& -
\begin{bmatrix}
\Delta x^T & \Delta s^T & 0 & 0
\end{bmatrix}
K_4
\begin{bmatrix}
\Delta x \\
\Delta s \\
\Delta y - \Delta_C c(x) \\
\Delta z - \Delta_G (g(x) + s)
\end{bmatrix} \\
=& -
\begin{bmatrix}
\Delta x \\
\Delta s
\end{bmatrix}^T
\begin{bmatrix}
P & 0 \\
0 & W^{-1}
\end{bmatrix}
\begin{bmatrix}
\Delta x \\
\Delta s
\end{bmatrix}
\\
&-
\begin{bmatrix}
C \Delta x \\
G \Delta x + \Delta s
\end{bmatrix}^T
\begin{bmatrix}
\Delta y - \Delta_C c(x) \\
\Delta z - \Delta_G (g(x) + s)
\end{bmatrix}.
\end{aligned}
\end{equation*}
By the last two block rows of the shifted linear system in
\cref{shifted-kkt-al-gradient-lemma},
\begin{equation*}
\begin{aligned}
\Delta y - \Delta_C c(x) &= \Delta_C C \Delta x, \\
\Delta z - \Delta_G(g(x)+s) &= \Delta_G(G\Delta x+\Delta s).
\end{aligned}
\end{equation*}
Substituting these identities gives
\begin{equation*}
\begin{aligned}
&D(\mathcal{A}(\cdot, \cdot, y, z; \mu, \Delta_C, \Delta_G);
(\Delta x, \Delta s))(x, s) \\
=& -
\begin{bmatrix}
\Delta x \\
\Delta s
\end{bmatrix}^T
K_{xs}
\begin{bmatrix}
\Delta x \\
\Delta s
\end{bmatrix}.
\end{aligned}
\end{equation*}
\end{proof}

\begin{theorem}
\label{inertia-al-descent-theorem}
If
\begin{equation*}
\begin{aligned}
\operatorname{In}(K_4) = (n_x+n_g,n_c+n_g,0),
\end{aligned}
\end{equation*}
then the solution of~\cref{ipm-4x4-newton-kkt} has primal component
$(\Delta x,\Delta s)$ satisfying either $(\Delta x,\Delta s)=0$ or
\begin{equation*}
\begin{aligned}
D(\mathcal{A}(\cdot, \cdot, y, z; \mu, \Delta_C, \Delta_G);
(\Delta x, \Delta s))(x, s) < 0 .
\end{aligned}
\end{equation*}
\end{theorem}

\begin{proof}
By~\cref{4x4-primal-inertia-lemma}, the assumed inertia of $K_4$ implies
$\operatorname{In}(K_{xs})=(n_x+n_g,0,0)$, and therefore $K_{xs}$ is positive
definite.
By~\cref{al-directional-derivative-lemma}, the directional derivative is
strictly negative whenever
$(\Delta x,\Delta s) \neq 0$.
\end{proof}

We can ensure that $K_4$ has the right inertia by sufficiently regularizing $P$,
typically by adding $\rho I$ to it for sufficiently large $\rho$.
This requires using a Newton-KKT factorization method that can determine
whether $K_4$ has the right inertia (but not necessarily determine its exact inertia in general).

Next, we give an example showing that, without dual regularization, the Newton-KKT solution
is not guaranteed to descend the augmented barrier-Lagrangian merit function.

\begin{example}
Consider the two-stage unconstrained linear-quadratic problem with states $(x_0, x_1)$ and controls $u_0$,
costs $\frac{1}{2}(x_0^2 + u_0^2 + x_1^2)$, and dynamics $x_1 = x_0 + u_0$, with initial state $s_0 = 0$.
If the current primal iterate is $(x_0, u_0, x_1) = (0, 0, 1)$ and the current dual iterate is
$(y_0, y_1) = (0, 3)$, the resulting Newton-KKT linear system is
\begin{equation*}
\begin{aligned}
\begin{bmatrix}
 1 & 0 &  0 & -1 &  1 \\
 0 & 1 &  0 &  0 &  1 \\
 0 & 0 &  1 &  0 & -1 \\
-1 & 0 &  0 &  0 &  0 \\
 1 & 1 & -1 &  0 &  0 \\
\end{bmatrix}
\begin{bmatrix}
\Delta x_0 \\
\Delta u_0 \\
\Delta x_1 \\
\Delta y_0\\
\Delta y_1
\end{bmatrix}
= \begin{bmatrix}
-3 \\
-3 \\
 2 \\
 0 \\
 1
\end{bmatrix},
\end{aligned}
\end{equation*}
and $(\Delta x_0, \Delta u_0, \Delta x_1, \Delta y_0, \Delta y_1) = (0, 0, -1, 0, -3)$ is the solution.
Thus, the directional derivative of the augmented barrier-Lagrangian merit function is
\begin{equation*}
\begin{aligned}
& -\frac{1}{2}
\begin{bmatrix}
\Delta x_0 & \Delta u_0 & \Delta x_1
\end{bmatrix}
\begin{bmatrix}
1 & 0 & 0 \\
0 & 1 & 0 \\
0 & 0 & 1 \\
\end{bmatrix}
\begin{bmatrix}
\Delta x_0 \\
\Delta u_0 \\
\Delta x_1
\end{bmatrix}
\\
& +
\begin{bmatrix}
0 & -1
\end{bmatrix}
\begin{bmatrix}
\Delta y_0 \\
\Delta y_1
\end{bmatrix}
+ \eta
\begin{bmatrix}
0 & -1
\end{bmatrix}
\begin{bmatrix}
-1 & 0 & 0 \\
 1 & 1 & -1 \\
\end{bmatrix}
\begin{bmatrix}
\Delta x_0 \\
\Delta u_0 \\
\Delta x_1
\end{bmatrix}
\\
& = 2 - \eta,
\end{aligned}
\end{equation*}
so $\eta > 2$ is required for the computed search direction to descend the merit function.
\end{example}

\subsection{Applying the regularized IPM to optimal control}\label{application}

As shown in~\cref{interior-point}, the line-search directions used in the
regularized interior-point method are computed by solving a linear system
of the form

\begin{equation*}
\begin{aligned}
\begin{bmatrix}
P & C^T & G^T \\
C & -\Delta_C^{-1} & 0 \\
G & 0 & -W - \Delta_G^{-1}
\end{bmatrix}
\begin{bmatrix}
\Delta x \\
\Delta y \\
\Delta z
\end{bmatrix} = -
\begin{bmatrix}
r_x \\
r_y \\
r_z
\end{bmatrix}.
\end{aligned}
\end{equation*}

Given that $\Delta z = (W + \Delta_G^{-1})^{-1}(G \Delta x + r_z)$,
we can also eliminate $\Delta z$, resulting in

\begin{equation*}
\begin{aligned}
& \begin{bmatrix}
P + G^T (W + \Delta_G^{-1})^{-1} G & C^T \\
C & -\Delta_C^{-1}
\end{bmatrix}
\begin{bmatrix}
\Delta x \\
\Delta y
\end{bmatrix}
\\ = & -
\begin{bmatrix}
r_x + G^T (W + \Delta_G^{-1})^{-1} r_z \\
r_y
\end{bmatrix}.
\end{aligned}
\end{equation*}

Any component of $\Delta y$ corresponding to an equality constraint other than
the dynamics can be eliminated in the same fashion. Importantly,
the only constraints that have cross-stage dependencies
are the dynamics, so these eliminations preserve the stagewise nature of the problem,
provided that $W, \Delta_C, \Delta_G$ are chosen so as not to introduce
cross-stage dependencies; $W, \Delta_C, \Delta_G$ are typically
diagonal, in which case this requirement is met.

This leaves us with a linear system that matches the first-order optimality conditions
of a dual-regularized LQR problem, as in~\cref{primal-dual-linear-system}.

Below, let $N, n, m$ denote the number of stages, states, and controls of this
dual-regularized LQR problem. Note that $n_x = N (n + m) + n$.

\begin{theorem}
$K_4$ has inertia $(n_x + n_g, n_c + n_g, 0)$ if and only if the left-hand side of
the first-order optimality conditions of the resulting dual-regularized LQR has inertia
$(N (n + m) + n, (N + 1) n, 0)$.
\end{theorem}
\begin{proof}
This follows by applying~\cref{sylvester-inertia-lemma} together with~\cref{3x3-2x2-inertia-lemma},
with $A$ chosen to cover the dynamics Jacobians (leaving the Jacobians of other equality constraints in $H$).
\end{proof}

The usual positive-semidefiniteness assumptions on the stagewise cost Hessians are
sufficient, but not necessary, for the dual-regularized LQR first-order
optimality matrix to have the correct inertia.

\begin{example}
Consider a scalar problem with
$N=1$,
\begin{equation*}
\begin{aligned}
Q_0=1,\quad M_0=0,\quad R_0=-1,\quad Q_1=1,
\end{aligned}
\end{equation*}
dynamics $x_1=2u_0$, and dual regularization $\Delta_0=\Delta_1=1$.
Then
\begin{equation*}
\begin{aligned}
P =
\begin{bmatrix}
1 & 0 & 0 \\
0 & -1 & 0 \\
0 & 0 & 1
\end{bmatrix},
\quad
C =
\begin{bmatrix}
-1 & 0 & 0 \\
0 & 2 & -1
\end{bmatrix},
\quad
\Delta = I .
\end{aligned}
\end{equation*}
The stage Hessian
\begin{equation*}
\begin{aligned}
\begin{bmatrix}
Q_0 & M_0 \\
M_0^T & R_0
\end{bmatrix}
=
\begin{bmatrix}
1 & 0 \\
0 & -1
\end{bmatrix}
\end{aligned}
\end{equation*}
is indefinite, and in particular $R_0$ is not positive-definite. Nevertheless,
the first-order optimality matrix
\begin{equation*}
\begin{aligned}
K_{\mathrm{LQR}} =
\begin{bmatrix}
P & C^T \\
C & -\Delta
\end{bmatrix}
\end{aligned}
\end{equation*}
has inertia $(3,2,0)$. Indeed, by~\cref{sylvester-inertia-lemma},
\begin{equation*}
\begin{aligned}
\operatorname{In}(K_{\mathrm{LQR}})
= \operatorname{In}(-I) + \operatorname{In}(P+C^TC),
\end{aligned}
\end{equation*}
and
\begin{equation*}
\begin{aligned}
P+C^TC =
\begin{bmatrix}
2 & 0 & 0 \\
0 & 3 & -2 \\
0 & -2 & 2
\end{bmatrix}
\succ 0,
\end{aligned}
\end{equation*}
with leading principal minors $2,6,4$. Thus
$\operatorname{In}(K_{\mathrm{LQR}})=(3,2,0)$, even though the usual
positive-semidefiniteness and positive-definiteness assumptions on the local
stage Hessian blocks are violated.
\end{example}

\begin{theorem}
\label{riccati-rational-identity-theorem}
Whenever the first-order optimality matrix in~\cref{primal-dual-linear-system} is
nonsingular and the Riccati recursions derived in~\cref{seq-algorithm-via-calculus}
and~\cref{parallel-algorithm} are well-defined, the
Riccati recursions return the unique solution of~\cref{primal-dual-linear-system},
even if the usual positive-semidefiniteness and positive-definiteness
assumptions on the local stage Hessian blocks are not satisfied.
\end{theorem}

\begin{proof}
For fixed dimensions, each component returned by the Riccati recursions is a
rational function of the problem data. Indeed, the quantities initialized
directly from the problem data are rational functions, and the field of
rational functions is closed under addition, multiplication, and division by
nonzero rational functions. It follows inductively that the entries of each
matrix formed during the recursion are rational functions of the original
problem data, and that each matrix inverse used by the recursion, whenever it
is defined, also has entries that are rational functions of the original
problem data. Similarly, whenever the first-order optimality matrix
$K_{\mathrm{LQR}}$ is nonsingular, each component of the unique solution
$-K_{\mathrm{LQR}}^{-1} [s^T \ c^T]^T$ is a rational function of the same data.

On the nonempty open set where the usual stronger assumptions hold, for
example when the local stage Hessian blocks and the matrices $\Delta_i$ are
positive-definite, the derivation of the Riccati recursions proves that these
two rational functions agree. The difference between any pair of corresponding
components is therefore a rational function that vanishes on a nonempty open
set. After clearing denominators, its numerator is a polynomial that vanishes
on a nonempty open set, and hence is identically zero. Thus the two rational
functions are identical.

Consequently, the Riccati expressions and the direct solution of
~\cref{primal-dual-linear-system} agree at every problem instance for which
both expressions are defined.
\end{proof}

\subsection{Associative Scans Overview}\label{associative-scans}

Associative scans are a common parallelization mechanism used in functional programming,
first introduced in~\cite{ascan}. They were used in~\cite{lqr-gpu} to derive a simple
method for solving (primal) LQR problems in $O(\log^2(m) + \log(N) \log^2(n))$ parallel time,
where $N, n, m$ are the number of stages, states, and controls, respectively.
Note, however, that~\cite{lqr-gpu} states that the complexity of their algorithm
is $O(m + \log(N) n)$, apparently stemming from some confusion around the parallel
time complexity of matrix inversion (which is $O(\log(n)^2)$ for $n \times n$ matrices,
as shown in~\cite{parallel-inverse}).

Given a set $\mathcal{X}$, a function $f: \mathcal{X} \times \mathcal{X} \rightarrow \mathcal{X}$
is said to be associative if $\forall a, b, c \in \mathcal{X}, f(f(a, b), c) = f(a, f(b, c))$.
The forward associative scan operation $S_f(x_1, \ldots, x_n; f)$ can be inductively defined by
\begin{equation*}
\begin{aligned}
S_f(x_1; f) &= x_1, \\
S_f(x_1, \ldots, x_{i + 1}; f) &= (y_1, \ldots, y_i, f(y_i, x_{i + 1})),
\end{aligned}
\end{equation*}
where $y_1, \ldots, y_i = S_f(x_1, \ldots, x_i; f)$.
Similarly, the reverse associative scan operation $S_r(x_1, \ldots, x_n; f)$ can be
inductively defined by
\begin{equation*}
\begin{aligned}
S_r(x_1; f) &= x_1, \\
S_r(x_1, \ldots, x_{i + 1}; f) &= (f(x_1, y_2), y_2, \ldots, y_{i + 1}),
\end{aligned}
\end{equation*}
where $y_2, \ldots, y_{i + 1} = S_r(x_2, \ldots, x_{i + 1}; f)$.
\cite{ascan} gives a method for performing associative scans of $N$ elements in parallel time $O(\log(N))$.

We use associative scans to derive the parallel dual-regularized Riccati recursion,
effectively extending~\cite{lqr-gpu}.

\section{Sequential Algorithm}\label{seq-algorithm-via-calculus}

We start with some simple lemmas
used to eliminate variables in~\cref{min-max-with-ys}.

\begin{lemma}
\label{eliminate-y}
If $M$ is symmetric and positive-definite and
$f(y) = k^T y - \frac{1}{2} y^T M y$, then
$\max \limits_y f(y) = f(M^{-1} k) = \frac{1}{2} \lVert k \rVert_{M^{-1}}^{2}$.
\end{lemma}

\begin{proof}
$\nabla f(y) = k - M y = 0 \Rightarrow y = M^{-1} k$.
Moreover, $f(M^{-1} k)
= k^T M^{-1} k - \frac{1}{2} k^T M^{-1} k
= \frac{1}{2} \lVert k \rVert_{M^{-1}}^{2}$.
\end{proof}

\begin{lemma}
\label{inverse-helper}
If $P$ is symmetric and positive semi-definite, and $M$ is symmetric and positive-definite,
\begin{equation*}
\begin{aligned}
I - (I + M P)^{-1} &= M P (I + M P)^{-1} = (I + M P)^{-1} M P , \\
(P + M^{-1})^{-1} &= (I + M P)^{-1} M = M (I + P M)^{-1} .
\end{aligned}
\end{equation*}
\end{lemma}

\begin{proof}
\begin{equation*}
\begin{aligned}
M P (I + M P)^{-1} & = (I + M P)(I + M P)^{-1} - (I + M P)^{-1} \\
& = I - (I + M P)^{-1} \\
(I + M P)^{-1} M P & = (I + M P)^{-1}(I + M P) - (I + M P)^{-1} \\
& = I - (I + M P)^{-1} \\
(P + M^{-1})^{-1} &= (M^{-1} (I + M P))^{-1} = (I + M P)^{-1} M \\
(P + M^{-1})^{-1} &= ((I + P M) M^{-1})^{-1} = M (I + P M)^{-1}.
\end{aligned}
\end{equation*}
\end{proof}

\begin{corollary}
\label{inverse-helper-corollary}
If $P$ and $M$ are symmetric and positive semi-definite,
\begin{equation*}
\begin{aligned}
(I + M P)^{-1} M = M (I + P M)^{-1} .
\end{aligned}
\end{equation*}
\end{corollary}

\begin{proof}
\begin{equation*}
\begin{aligned}
& (I + M P)^{-1} M = \lim \limits_{\epsilon \downarrow 0} (I + (M + \epsilon I) P)^{-1} (M + \epsilon I) \\
= & \lim \limits_{\epsilon \downarrow 0} (M + \epsilon I) (I + P (M + \epsilon I))^{-1}
= M (I + P M)^{-1}.
\end{aligned}
\end{equation*}
\end{proof}

\begin{lemma}
\label{eliminate-x}
If $P$ is symmetric and positive semi-definite,
$M$ is symmetric and positive-definite, and
$f(x) = \frac{1}{2} x^T P x + p^T x + \frac{1}{2} \lVert c - x \rVert_{M^{-1}}^2$, then
\begin{equation*}
\begin{aligned}
\min \limits_x f(x) = & f \left( (I + M P)^{-1} (c - M p) \right) \\
= & \frac{1}{2} c^T P (I + M P)^{-1} c - \frac{1}{2}p^T (I + M P)^{-1} M p \\
+ & p^T c - p^T M P(I + M P)^{-1} c.
\end{aligned}
\end{equation*}
\end{lemma}

\begin{proof}
Noting that
\begin{equation*}
\begin{aligned}
f(x) = \frac{1}{2} x^T (P + M^{-1}) x + (p - M^{-1} c)^T x + \frac{1}{2} c^T M^{-1} c,
\end{aligned}
\end{equation*}
it follows that
\begin{equation*}
\begin{aligned}
& \nabla f(x) = (P + M^{-1}) x + p - M^{-1} c = 0 \\
\Rightarrow & (P + M^{-1}) x = M^{-1} c - p \\
\Rightarrow & x = (P + M^{-1})^{-1} (M^{-1} c - p).
\end{aligned}
\end{equation*}

Moreover, employing~\cref{inverse-helper},
\begin{equation*}
\begin{aligned}
& f \left( (P + M^{-1})^{-1} (M^{-1} c - p) \right) \\
= & \frac{1}{2} (M^{-1} c - p)^T (P + M^{-1})^{-1} (M^{-1} c - p) \\
+ & (p - M^{-1} c)^T (P + M^{-1})^{-1} (M^{-1} c - p) + \frac{1}{2} c^T M^{-1} c \\
= & -\frac{1}{2} (M^{-1} c - p)^T (P + M^{-1})^{-1} (M^{-1} c - p) + \frac{1}{2} c^T M^{-1} c \\
= & \frac{1}{2} c^T (M^{-1} - M^{-1} (P + M^{-1})^{-1} M^{-1}) c \\
- & \frac{1}{2} p^T (P + M^{-1})^{-1} p + p^T ((P + M^{-1})^{-1} M^{-1}) c \\
= & \frac{1}{2} c^T M^{-1} (I - (I + M P)^{-1}) c \\
- & \frac{1}{2} p^T (I + M P)^{-1} M p + p^T (I + M P)^{-1} c \\
  = & \frac{1}{2} c^T M^{-1} M P (I + M P)^{-1} c - \frac{1}{2} p^T (I + M P)^{-1} M p \\
+ & p^T c - p^T (I - (I + M P)^{-1}) c \\
= & \frac{1}{2} c^T P (I + M P)^{-1} c - \frac{1}{2} p^T (I + M P)^{-1} M p \\
+ & p^T c - p^T (M P (I + M P)^{-1}) c .
\end{aligned}
\end{equation*}

Finally, note that
\begin{equation*}
\begin{aligned}
(P + M^{-1})^{-1} (M^{-1} c - p) &= (I + M P)^{-1} M (M^{-1} c - p) \\
&= (I + M P)^{-1} (c - M p) .
\end{aligned}
\end{equation*}
\end{proof}

\begin{definition}
For $k \in \lbrace 0, \ldots, N \rbrace$, we define the cost-to-go functions
\begin{equation}
\begin{aligned}
\label{cost-to-go-value-functions}
V_{k}(x_k) =
& \max \limits_{y_{k+1}, \ldots, y_N} \min \limits_{u_k, x_{k+1}, \ldots, u_{N-1}, x_N} \\
& \sum \limits_{i=k}^{N-1}
\left( \frac{1}{2} x_i^T Q_i x_i + x_i^T M_i u_i + \frac{1}{2} u_i^T R_i u_i \right. \\
& \left. \vphantom{\frac{1}{1}}+ q_i^T x_i + r_i^T u_i \right)
+ \frac{1}{2} x_N^T Q_N x_N + q_N^T x_N \\
& + \sum \limits_{i=k}^{N-1} \left( y_{i+1}^T \left( A_i x_i
+ B_i u_i + c_{i + 1} - x_{i + 1} \right) \right) \\
  & - \sum \limits_{i=k}^{N-1} \frac{1}{2} y_{i+1}^T \Delta_{i+1} y_{i+1} . \\
\end{aligned}
\end{equation}
\end{definition}

This brings us to the key theorem.

\begin{theorem}
\label{main-seq-theorem}
For all $k \in \lbrace 0, \ldots, N \rbrace$, there exist $P_k, p_k$ such
that
\begin{equation*}
\begin{aligned}
 V_k(x_k) = \frac{1}{2} x_k^T P_k x_k + p_k^T x_k + V_k(0),
\end{aligned}
\end{equation*}
where $P_k$ is symmetric and positive semi-definite,
as well as $K_k, k_k$ such that~\cref{min-max-with-ys} is optimal at
\begin{equation*}
\begin{aligned}
u_k =& K_k x_k + k_k , \\
x_0 =& (I + \Delta_0 P_0)^{-1} (c_0 - \Delta_0 p_0) , \\
x_{k+1} =& (I + \Delta_{k+1} P_{k+1})^{-1} \\
& (A_k x_k + B_k u_k + c_{k+1} - \Delta_{k+1} p_{k+1}) , \\
y_k =& P_k x_k + p_k .
\end{aligned}
\end{equation*}
\end{theorem}

\begin{proof}
We proceed by induction, in decreasing order of $k$.
The base case, consisting of $k = N$, holds trivially by setting
$P_N = Q_N$ and $p_N = q_N$.

Assuming, by the induction hypothesis, that the statement holds true for $k+1$,
we will show it remains true for $k$.

Observing that
\begin{equation}
\begin{aligned}
\label{value-function-recursive-definition}
  V_k(x_k) = & \max \limits_{y_{k+1}} \min \limits_{u_k, x_{k+1}} V_{k+1} (x_{k+1})
  + \frac{1}{2} x_k^T Q_k x_k \\
	& + x_k^T M_k u_k + \frac{1}{2} u_k^T R_k u_k
  - \frac{1}{2} y_{k+1}^T \Delta_{k+1} y_{k+1} \\
  & + y_{k+1}^T (A_k x_k + B_k u_k + c_{k+1} - x_{k+1}) ,
\end{aligned}
\end{equation}
we start by eliminating $y_{k+1}$ from~\cref{value-function-recursive-definition},
using~\cref{eliminate-y}.

Note that
\begin{equation*}
\begin{aligned}
& \max \limits_{y_{k+1}}
y_{k+1}^T \left( A_k x_k + B_k u_k + c_{k+1} - x_{k+1} \right) \\
& \quad - \frac{1}{2} y_{k+1}^T \Delta_{k+1} y_{k+1} \\
& = \frac{1}{2}
\lVert A_k x_k + B_k u_k + c_{k+1} - x_{k+1}
\rVert_{\Delta_{k+1}^{-1}}^2 ,
\end{aligned}
\end{equation*}
achieved at
\begin{equation}
\begin{aligned}
\label{recover-y}
y_{k+1} = \Delta_{k+1}^{-1}
(A_k x_k + B_k u_k + c_{k+1} - x_{k+1}) .
\end{aligned}
\end{equation}

Thus,
\begin{equation*}
\begin{aligned}
  V_{k} (x_{k}) = & \min \limits_{u_k, x_{k+1}} V_{k+1} (x_{k+1})
+ \frac{1}{2} x_k^T Q_k x_k + x_k^T M_k u_k \\
& + \frac{1}{2} \lVert A_k x_k + B_k u_k + c_{k+1} - x_{k+1} \rVert_{\Delta_{k+1}^{-1}}^2 \\
& + \frac{1}{2} u_k^T R_k u_k + \text{const} .
\end{aligned}
\end{equation*}

Applying the induction hypothesis,
\begin{equation}
\begin{aligned}
\label{before-eliminating-x}
& V_{k} (x_{k}) = \min \limits_{u_k, x_{k+1}}
\frac{1}{2} x_k^T Q_k x_k + x_k^T M_k u_k + \frac{1}{2} u_k^T R_k u_k \\
& + \frac{1}{2} x_{k+1}^T P_{k+1} x_{k+1} + p_{k+1}^T x_{k+1} \\
& + \frac{1}{2} \lVert A_k x_k + B_k u_k + c_{k+1} - x_{k+1} \rVert_{\Delta_{k+1}^{-1}}^2
+ \text{const} .
\end{aligned}
\end{equation}

Next, we apply~\cref{eliminate-x} to eliminate $x_{k+1}$ from~\cref{before-eliminating-x}.

The terms involving $x_{k+1}$ are
\begin{equation}
\begin{aligned}
\label{next-x-terms}
\min \limits_{x_{k+1}} & \frac{1}{2} x_{k+1}^T P_{k+1} x_{k+1} + p_{k+1}^T x_{k+1} \\
& + \frac{1}{2} \lVert A_k x_k + B_k u_k + c_{k+1} - x_{k+1} \rVert_{\Delta_{k+1}^{-1}}^2.
\end{aligned}
\end{equation}

Letting
$W_{k+1} = P_{k+1} (I + \Delta_{k+1} P_{k+1})^{-1}$,
it follows from~\cref{inverse-helper-corollary} that $W_{k+1} = W_{k+1}^T$.

Applying~\cref{eliminate-x},
and discarding additive constant terms,~\cref{next-x-terms} becomes
\begin{equation*}
\begin{aligned}
& \frac{1}{2} (A_k x_k + B_k u_k + c_{k+1})^T W_{k+1}
(A_k x_k + B_k u_k + c_{k+1}) \\
& + p_{k+1}^T (A_k x_k + B_k u_k + c_{k+1}) \\
& - p_{k+1}^T \Delta_{k+1} W_{k+1} (A_k x_k + B_k u_k + c_{k+1}) ,
\end{aligned}
\end{equation*}
achieved at
\begin{equation}
\begin{aligned}
\label{recover-x}
x_{k+1} = & (I + \Delta_{k+1} P_{k+1})^{-1} \\
& (A_k x_k + B_k u_k + c_{k+1} - \Delta_{k+1} p_{k+1}) .
\end{aligned}
\end{equation}

Next, we eliminate $u_k$ from the resulting version of~\cref{before-eliminating-x}
after $x_{k+1}$ has been eliminated. The terms involving $u_k$ are
\begin{equation}
\begin{aligned}
\label{only-u-terms}
& \min \limits_{u_k} (r_k + M_k^T x_k)^T u_k + \frac{1}{2} u_k^T R_k u_k \\
& + \frac{1}{2} (A_k x_k + B_k u_k + c_{k+1})^T W_{k+1} (A_k x_k + B_k u_k + c_{k+1}) \\
& + p_{k+1}^T B_k u_k
- p_{k+1}^T \Delta_{k+1} W_{k+1} B_k u_k \\
& = \min \limits_{u_k} \frac{1}{2} u_k^T
(R_k + B_k^T W_{k+1} B_k) u_k \\
& + \left( r_k + M_k^T x_k + B_k^T \left( W_{k+1} (A_k x_k + c_{k+1}) \right. \right. \\
& \left. \left. + p_{k+1} - W_{k+1} \Delta_{k+1} p_{k+1} \right) \right)^T u_k \\
& + \frac{1}{2} (A_k x_k + c_{k+1})^T W_{k+1}
(A_k x_k + c_{k+1}) .
\end{aligned}
\end{equation}

Note that this leaves out the terms
\begin{equation}
\begin{aligned}
\label{left-out-x-terms}
p_{k+1}^T (A_k x_k + c_{k+1}) - p_{k+1}^T \Delta_{k+1} W_{k+1} (A_k x_k + c_{k+1}) ,
\end{aligned}
\end{equation}
which do not depend on $u_k$. We add them back later.

Letting
\begin{equation*}
\begin{aligned}
G_k &= R_k + B_k^T W_{k+1} B_k \\
g_{k+1} &= p_{k+1} + W_{k+1} (c_{k+1} - \Delta_{k+1} p_{k+1}) \\
H_k &= B_k^T W_{k+1} A_k + M_k^T \\
h_k &= r_k + B_k^T g_{k+1} \\
K_k &= -G_k^{-1} H_k \\
k_k &= -G_k^{-1} h_k ,
\end{aligned}
\end{equation*}
\cref{only-u-terms} becomes
\begin{equation*}
\begin{aligned}
\min \limits_{u_k} & \frac{1}{2} u_k^T G_k u_k + (H_k x_k + h_k)^T u_k \\
& + \frac{1}{2} (A_k x_k + c_{k+1})^T W_{k+1} (A_k x_k + c_{k+1}) .
\end{aligned}
\end{equation*}

Taking the gradient with respect to $u_{k}$ and equating it to $0$, we get
\begin{equation}
\begin{aligned}
\label{control-law}
u_{k} = - G_k^{-1} (H_k x_k + h_k) = K_k x_k + k_k.
\end{aligned}
\end{equation}

Plugging this back in, and discarding additive constants, we get
\begin{equation}
\begin{aligned}
\label{after-eliminating-u}
& -\frac{1}{2} (H_k x_k + h_k)^T G_k^{-1} (H_k x_k + h_k) \\
& + \frac{1}{2} (A_k x_k + c_{k+1})^T W_{k+1} (A_k x_k + c_{k+1})
\\ = & \frac{1}{2} x_k^T (A_k^T W_{k+1} A_k - H_k^T G_k^{-1} H_k) x_k \\
& + (A_k^T W_{k+1} c_{k+1} - H_k^T G_k^{-1} h_k)^T x_k \\
= & \frac{1}{2} x_k^T (A_k^T W_{k+1} A_k + H_k^T K_k) x_k \\
& + (A_k^T W_{k+1} c_{k+1} + H_k^T k_k)^T x_k .
\end{aligned}
\end{equation}

Collecting all remaining terms from~\cref{left-out-x-terms}
and~\cref{after-eliminating-u}, discarding additive constants, and letting
\begin{equation}
\begin{aligned}
\label{value-fn-formulas}
P_k &= A_k^T W_{k+1} A_k + Q_k + H_k^T K_k \\
p_k &= q_k + A_k^T g_{k+1} + H_k^T k_k , \\
\end{aligned}
\end{equation}
~\cref{before-eliminating-x} becomes
\begin{equation*}
\begin{aligned}
V_k(x_k) & = \frac{1}{2} x_k^T P_k x_k + p_k^T x_k + \text{const} \\
& = \frac{1}{2} x_k^T P_k x_k + p_k^T x_k + V_k(0) .
\end{aligned}
\end{equation*}

Next, we must show that $P_k$ remains positive semi-definite.

Noting that $P_k$ is the Schur complement of
\begin{equation}
\begin{aligned}
\label{psd-block-matrices}
& \begin{bmatrix}
Q_k + A_k^T W_{k+1} A_k  &  H_k^T \\
H_k                      &  G_k
\end{bmatrix} \\
= & \begin{bmatrix}
Q_k    &  M_k^T \\
M_k^T  &  R_k
\end{bmatrix} +
\begin{bmatrix}
A_k^T W_{k+1} A_k  &  A_k^T W_{k+1} B_k  \\
B_k^T W_{k+1} A_k  &  B_k^T W_{k+1} B_k
\end{bmatrix} \\
= & \begin{bmatrix}
Q_k    &  M_k^T \\
M_k^T  &  R_k
\end{bmatrix} +
\begin{bmatrix}
A_k^T \\
B_k^T
\end{bmatrix}
W_{k+1}
\begin{bmatrix}
A_k & B_k
\end{bmatrix}
,
\end{aligned}
\end{equation}
it follows that $P_k$ is positive semi-definite, as $G_k$ is positive-definite
and the block matrix~\cref{psd-block-matrices} is positive semi-definite.

Noting that~\cref{min-max-with-ys} can be written as
\begin{equation*}
\begin{aligned}
\max \limits_{y_0} \min \limits_{x_0}
V_0(x_0) + y_0^T (x_0 - c_0) - \frac{1}{2} y_0^T \Delta_0 y_0 ,
\end{aligned}
\end{equation*}
we can apply~\cref{eliminate-y} to rewrite this as
\begin{equation*}
\begin{aligned}
\min \limits_{x_0}
& V_0(x_0) + \frac{1}{2} \lVert c_0 - x_0 \rVert_{\Delta_0^{-1}}^2 \\
= \min \limits_{x_0}
& \frac{1}{2} x_0^T (P_0 + \Delta_0^{-1}) x_0
+ (p_0 - \Delta_0^{-1} c_0)^T x_0 \\
& + \frac{1}{2} \lVert c_0 \rVert_{\Delta_0}^2 + V_0(0),
\end{aligned}
\end{equation*}
achieved at
\begin{equation}
\begin{aligned}
\label{recover-y0}
y_0 = \Delta_0^{-1} (c_0 - x_0) .
\end{aligned}
\end{equation}

Solving for $x_0$, we get
\begin{equation}
\begin{aligned}
\label{recover-x0}
x_0 & = (P_0 + \Delta_0^{-1})^{-1} (\Delta_0^{-1} c_0 - p_0) \\
& = (I + \Delta_0 P_0)^{-1} \Delta_0 (\Delta_0^{-1} c_0 - p_0) \\
& = (I + \Delta_0 P_0)^{-1} (c_0 - \Delta_0 p_0).
\end{aligned}
\end{equation}

The remaining $x_k, u_k$ can be recovered in a forward pass,
i.e. in increasing order of $k$ via~\cref{recover-x} and~\cref{control-law}.

Finally, we can rewrite~\cref{recover-y0} and~\cref{recover-y}
to avoid using $\Delta_0^{-1}, \ldots, \Delta_N^{-1}$,
in the interest of improving numerical stability when
$\Delta_0, \ldots, \Delta_N \rightarrow 0$ as well as
extending this method to the case of $\Delta_0, \ldots, \Delta_N = 0$
(recovering the standard backward and forward passes of the standard LQR algorithm).

Since, due to~\cref{recover-x0},
\begin{equation*}
\begin{aligned}
(P_0 + \Delta_0^{-1}) x_0 &= \Delta_0^{-1} c_0 - p_0 \\
\Leftrightarrow \Delta_0^{-1} (c_0 - x_0) &= P_0 x_0 + p_0,
\end{aligned}
\end{equation*}
we can rewrite~\cref{recover-y0} as
\begin{equation*}
\begin{aligned}
y_0 = P_0 x_0 + p_0.
\end{aligned}
\end{equation*}

Moreover, due to~\cref{recover-x},
\begin{equation*}
\begin{aligned}
(I + \Delta_{k+1} P_{k+1}) x_{k+1} = A_k x_k + B_k u_k + c_{k+1} - \Delta_{k+1} p_{k+1}
\\ \Leftrightarrow \Delta_{k+1}^{-1} (A_k x_k + B_k u_k + c_{k+1} - x_{k+1})
= P_{k+1} x_{k+1} + p_{k+1} .
\end{aligned}
\end{equation*}
Thus, we can rewrite~\cref{recover-y} as
\begin{equation*}
\begin{aligned}
y_{k+1} = P_{k+1} x_{k+1} + p_{k+1} .
\end{aligned}
\end{equation*}
\end{proof}

The following result is used in~\cite{rlqr-jax} to accelerate
multiple problem solves sharing the same left-hand side.
We will also need it in~\cref{parallel-algorithm}.

\begin{corollary}
\label{p-recurrence}
The variables $p_k$ satisfy the affine recurrence
\begin{equation*}
\begin{aligned}
p_k = & q_k + K_k^T r_k \\
+ & (A_k + B_k K_k)^T (p_{k+1} + W_{k+1} (c_{k+1} - \Delta_{k+1} p_{k+1})) .
\end{aligned}
\end{equation*}
\end{corollary}

\begin{proof}
Noting that
\begin{equation*}
\begin{aligned}
K_k^T h_k = -(G_k^{-1} H_k)^T h_k = H_k^T (-G_k^{-1} h_k) = H_k^T k_k,
\end{aligned}
\end{equation*}
it follows that
\begin{equation*}
\begin{aligned}
p_k = & q_k + A_k^T g_{k+1} + H_k^T k_k = q_k + A_k^T g_{k+1} + K_k^T h_k \\
= & q_k + A_k^T g_{k+1} + K_k^T (r_k + B_k^T g_{k+1}) \\
= & q_k + K_k^T r_k + (A_k + B_k K_k)^T g_{k+1} = q_k + K_k^T r_k \\
+ & (A_k + B_k K_k)^T (p_{k+1} + W_{k+1} (c_{k+1} - \Delta_{k+1} p_{k+1})) .
\end{aligned}
\end{equation*}
\end{proof}

\section{Parallel Algorithm}
\label{parallel-algorithm}

We first eliminate the control variables $u_i$ from~\cref{min-max-with-ys}.
Computing the gradients with respect to $u_i$, equating them to $0$,
and solving for $u_i$ yields
\begin{equation*}
\begin{aligned}
u_i = -R_i^{-1} \left( M_i^T x_i + r_i + B_i^T y_{i+1} \right).
\end{aligned}
\end{equation*}

Noting that
\begin{equation*}
\begin{aligned}
\frac{1}{2} u_i^T R_i u_i = & \frac{1}{2} x_i^T M_i R_i^{-1} M_i^T x_i +
r_i^T R_i^{-1} M_i^T x_i \\
& + \frac{1}{2} y_{i+1}^T B_i R_i^{-1} B_i^T y_{i+1}
+ y_{i+1}^T B_i R_i^{-1} M_i^T x_i \\
& + r_i^T R_i^{-1} B_i^T y_{i+1}
+ \frac{1}{2} r_i^T R_i^{-1} r_i, \\
x_i^T M_i u_i = & -x_i^T M_i R_i^{-1} M_i^T x_i - r_i^T R_i^{-1} M_i^T x_i \\
& - y_{i+1}^T B_i R_i^{-1} M_i^T x_i, \\
r_i^T u_i = & -r_i^T R_i^{-1} M_i^T x_i - r_i^T R_i^{-1} r_i \\
& - y_{i+1}^T B_i R_i^{-1} r_i, \\
y_{i+1}^T B_i u_i = & -y_{i+1}^T B_i R_i^{-1} B_i^T y_{i+1} \\
& - y_{i+1}^T B_i R_i^{-1} M_i^T x_i
- r_i^T R_i^{-1} B_i^T y_{i+1},
\end{aligned}
\end{equation*}
and discarding constant terms, the optimization problem~\cref{min-max-with-ys} becomes
\begin{equation}
\begin{aligned}
\label{min-max-without-us}
& \max \limits_{y_0, \ldots, y_N} \min \limits_{x_0, \ldots, x_N}
\sum \limits_{i=0}^{N-1}
\left( \frac{1}{2} x_i^T (Q_i - M_i R_i^{-1} M_i^T) x_i \right. \\
& \left. \vphantom{\frac{1}{1}}+ (q_i - M_i R_i^{-1} r_i)^T x_i \right)
+ \frac{1}{2} x_N^T Q_N x_N + q_N^T x_N \\
& + \sum \limits_{i=0}^{N-1} \left( y_{i+1}^T \left( \left( A_i - B_i R_i^{-1} M_i^T \right) x_i \right. \right. \\
& \left. \left. + \left( c_{i+1} - B_i R_i^{-1} r_i \right) - x_{i + 1} \right) \right) \\
& - \frac{1}{2} \sum \limits_{i=0}^{N-1} y_{i+1}^T
\left( \Delta_{i+1} + B_i R_i^{-1} B_i^T \right) y_{i+1} \\
& + y_0^T (c_0 - x_0) - \frac{1}{2} y_0^T \Delta_0 y_0 .
\end{aligned}
\end{equation}

For $i \in \lbrace 0, \ldots, N - 1 \rbrace$, we let
\begin{equation}
\begin{aligned}
  \label{lqr-gpu-init-running}
  P_{i \rightarrow i+1} &= Q_i - M_i R_i^{-1} M_i^T, \\
  p_{i \rightarrow i+1} &= q_i - M_i R_i^{-1} r_i, \\
  A_{i \rightarrow i+1} &= A_i - B_i R_i^{-1} M_i^T, \\
  C_{i \rightarrow i+1} &= \Delta_{i+1} + B_i R_i^{-1} B_i^T, \\
  c_{i \rightarrow i+1} &= c_{i+1} - B_i R_i^{-1} r_i.
\end{aligned}
\end{equation}

Similarly, we let
\begin{equation}
\begin{aligned}
  \label{lqr-gpu-init-terminal}
  P_{N \rightarrow N+1} &= Q_N, \\
  p_{N \rightarrow N+1} &= q_N, \\
  A_{N \rightarrow N+1} &= 0, \\
  C_{N \rightarrow N+1} &= 0, \\
  c_{N \rightarrow N+1} &= 0.
\end{aligned}
\end{equation}

\begin{definition}
For $0 \leq i < j \leq N+1$, we define the interval value functions
\begin{equation*}
\begin{aligned}
V_{i \rightarrow j}(x_i, x_j) = &
\max \limits_{y_{i+1}, \ldots, y_j} \min \limits_{x_{i+1}, \ldots, x_{j-1}} \\
& \sum \limits_{k=i}^{j-1}
\left( \frac{1}{2} x_k^T P_{k \rightarrow k + 1} x_k
+ p_{k \rightarrow k+1}^T x_k \right) \\
+ & \sum \limits_{k=i}^{j-1} \left( y_{k+1}^T \left( A_{k \rightarrow k+1} x_k
+ c_{k \rightarrow k+1} - x_{k + 1} \right) \right) \\
- & \frac{1}{2} \sum \limits_{k=i}^{j-1} y_{k+1}^T C_{k \rightarrow k+1} y_{k+1} .
\end{aligned}
\end{equation*}
\end{definition}

Note that, by definition,
\begin{equation}
\begin{aligned}
V_i(x_i) = \min \limits_{x_{N+1}} V_{i \rightarrow N+1} (x_i, x_{N+1}).
\end{aligned}
\end{equation}

Since
\begin{equation*}
\begin{aligned}
\max \limits_{y_{N+1}} \min \limits_{x_{N+1}} -y_{N+1}^T x_{N+1} = 0,
\end{aligned}
\end{equation*}
achieved by forcing $x_{N+1} = 0$, it follows that
\begin{equation}
\begin{aligned}
V_i(x_i) = V_{i \rightarrow N+1} (x_i, 0).
\end{aligned}
\end{equation}

It was shown in~\cite{lqr-gpu} that the functions $V_{i \rightarrow j}$
admit representations of the form
\begin{equation*}
\begin{aligned}
V_{i \rightarrow j} (x_i, x_j) = \max \limits_{y} \left(
  \frac{1}{2} x_i^T P_{i \rightarrow j} x_i + p_{i \rightarrow j}^T x_i
  \right. \\
	\left. - \frac{1}{2} y^T C_{i \rightarrow j} y
  + y^T \left( A_{i \rightarrow j} x_i + c_{i \rightarrow j} - x_j \right) \right),
\end{aligned}
\end{equation*}
modulo constant additive terms that are independent of $x_i, x_j$ and can thus
be discarded.

The following combination rules, established in~\cite{lqr-gpu},
can be applied to compute $V_{i \rightarrow k}$
from $V_{i \rightarrow j}$ and $V_{j \rightarrow k}$:
\begin{equation}
\begin{aligned}
  \label{lqr-gpu-combo}
  P_{i \rightarrow k} =& A_{i \rightarrow j}^T \left( I + P_{j \rightarrow k} C_{i \rightarrow j} \right)^{-1}
	  P_{j \rightarrow k} A_{i \rightarrow j} + P_{i \rightarrow j}, \\
  p_{i \rightarrow k} =& A_{i \rightarrow j}^T \left( I + P_{j \rightarrow k} C_{i \rightarrow j} \right)^{-1}
	  \left( p_{j \rightarrow k} + P_{j \rightarrow k} c_{i \rightarrow j} \right) \\
    & + p_{i \rightarrow j}, \\
  A_{i \rightarrow k} =& A_{j \rightarrow k} \left( I + C_{i \rightarrow j}
	  P_{j \rightarrow k} \right)^{-1} A_{i \rightarrow j}, \\
  C_{i \rightarrow k} =& A_{j \rightarrow k} \left( I + C_{i \rightarrow j} P_{j \rightarrow k} \right)^{-1}
	  C_{i \rightarrow j} A_{j \rightarrow k}^T + C_{j \rightarrow k}, \\
  c_{i \rightarrow k} =& A_{j \rightarrow k} \left( I + C_{i \rightarrow j} P_{j \rightarrow k} \right)^{-1}
	  \left( c_{i \rightarrow j} - C_{i \rightarrow j} p_{j \rightarrow k} \right) \\
    & + c_{j \rightarrow k}. \\
\end{aligned}
\end{equation}

A reverse associative scan~\cite{ascan} can be used to compute the $P_{i \rightarrow N+1}, p_{i \rightarrow N+1}$
(i.e. the $P_i, p_i$ in~\cref{value-fn-formulas}) in $O(\log(N) \log (n)^2)$ via~\cref{lqr-gpu-combo}.
However, it is preferable to only include the computation of the $P_{i \rightarrow k}, A_{i \rightarrow k}, C_{i \rightarrow k}$
in the reverse associative scan (as the bare minimum required to extract the $P_{i \rightarrow N+1}$),
recovering the $p_{i \rightarrow N+1}$ at the end via~\cref{p-recurrence}.

Moreover, due to~\cref{inverse-helper-corollary}, we can also write
\begin{equation}
\begin{aligned}
P_{i \rightarrow k} =& A_{i \rightarrow j}^T P_{j \rightarrow k}
\left( I + C_{i \rightarrow j} P_{j \rightarrow k} \right)^{-1} A_{i \rightarrow j}
+ P_{i \rightarrow j}, \\
\end{aligned}
\end{equation}
resulting in a single matrix inversion per step, which provides a substantial speedup over~\cite{lqr-gpu}.

Next, the matrices $(I + \Delta_i P_i)^{-1}$ can be computed in $O(\log(n)^2)$ parallel time (i.e. $O(1)$ with respect to $N$).

We can then compute $x_0$ in $O(\log(n))$ parallel time via~\cref{recover-x0}
and the $K_i, k_i$ from~\cref{control-law} in $O(\log(n) + \log (m)^2)$ parallel time
(i.e. $O(1)$ with respect to $N$ in both cases).

Finally, we compute the $u_i, x_{i+1}$ from the $K_i, k_i$ via~\cref{recover-x} and~\cref{control-law}.
Note that the sequential LQR forward pass has $O(N)$ parallel time complexity.
However, as done in~\cite{lqr-gpu}, we can reduce the computation of
the $x_i$ to a sequential composition of affine functions, which can also be parallelized with an associative scan~\cite{ascan}
in $O(\log(N) \log(n))$ parallel time.
This is described in~\cref{composing-affine-functions}. By~\cref{main-seq-theorem},

\begin{equation*}
\begin{aligned}
x_{i+1} = & (I + \Delta_{i+1} P_{i+1})^{-1} \left( A_i x_i + B_i u_i + c_{i+1} - \Delta_{i+1} p_{i+1} \right) \\
= & (I + \Delta_{i+1} P_{i+1})^{-1} \\
& \left( A_i x_i + B_i \left( K_i x_i + k_i \right) + c_{i+1} - \Delta_{i+1} p_{i+1} \right) \\
= & (I + \Delta_{i+1} P_{i+1})^{-1} \left( A_i + B_i K_i \right) x_i \\
& + (I + \Delta_{i+1} P_{i+1})^{-1} ( B_i k_i + c_{i+1} - \Delta_{i+1} p_{i+1}) .
\end{aligned}
\end{equation*}

Once these affine functions have been composed, they can be independently applied to $x_0$ to recover all the $x_i$
in $O(\log(n))$ parallel time.
The $u_i$ can then be computed in $O(\log (m) + \log (n))$ parallel time
by independently evaluating $u_i = K_i x_i + k_i$.
Note that both of these parallel times are $O(1)$ with respect to $N$.

Thus, the combined parallel time complexity of our method is $O(\log(m)^2 + \log(N) \log(n)^2)$.

\subsection{Composing Affine Functions with Associative Scans}
\label{composing-affine-functions}

In this section, we describe an algorithm for composing
$N$ affine functions $F_i(x) = M_i x + m_i$, as done in~\cite{lqr-gpu},
with $O(\log(N) \log(n))$ parallel time.

Letting $\mathcal{X} = \mathbb{R}^{n} \times \mathbb{R}^{n \times n}$ and letting
$f: \mathcal{X} \times \mathcal{X} \rightarrow \mathcal{X}$ be defined by $f((a, B), (c, D)) = (Da + c, DB)$,
we claim that $f$ is associative. This is simple to verify:
\begin{equation*}
\begin{aligned}
  &f(f((a, B), (c, D)), (e, F)) = f((Da + c, DB), (e, F)) \\
= &(F(Da + c) + e, F(DB)) = ((FD)a + Fc + e, (FD)B) \\
= &f((a, B), (Fc + e, FD)) = f((a, B), f((c, D), (e, F))).
\end{aligned}
\end{equation*}

Moreover, note that $f$ is the affine function composition operator, as
$D(Bx + a) + c = (DB)x + (Da + c)$.
Thus, it suffices to apply a forward associative scan to $f$ to obtain all compositions
$F_0, F_1 \circ F_0, \ldots, F_{N-1} \circ \cdots \circ F_0$ in $O(\log(N) \log(n))$
parallel time.

\section{Inertia Certification}
\label{inertia-certification}

The inertia of the dual-regularized LQR first-order optimality matrix can be
certified using only the small matrices factored by the Riccati recursion. In
the sequential recursion, define the symmetric state pivots
\begin{equation*}
\begin{aligned}
S_i = P_i+\Delta_i^{-1}, \qquad i=0,\ldots,N,
\end{aligned}
\end{equation*}
where $P_i$ are the value-function Hessians from~\cref{value-fn-formulas}, and
let $G_i$ be the control pivots from~\cref{control-law}.
The matrices $I+\Delta_i P_i$ that appear in the Riccati formulas are generally
not symmetric; they are the left-scaled state pivots, since
$I+\Delta_i P_i = \Delta_i S_i$. For numerical certification, one may instead
use the symmetric, inverse-free representative
\begin{equation*}
\begin{aligned}
\widehat S_i
=\Delta_i^{1/2}S_i\Delta_i^{1/2}
= I+\Delta_i^{1/2}P_i\Delta_i^{1/2}.
\end{aligned}
\end{equation*}
The nonsymmetric matrix $I+\Delta_i P_i$ is similar to $\widehat S_i$:
\begin{equation*}
\begin{aligned}
I+\Delta_i P_i
= \Delta_i^{1/2} \widehat S_i \Delta_i^{-1/2}.
\end{aligned}
\end{equation*}
Thus, invertibility of $I+\Delta_i P_i$ is equivalent to invertibility of
$S_i$ and $\widehat S_i$, while positive-definiteness may be checked on either
symmetric matrix $S_i$ or $\widehat S_i$.

\begin{theorem}
\label{riccati-inertia-certification-theorem}
Assume $\Delta_i \succ 0$ for $i=0,\ldots,N$. The left-hand side of
the dual-regularized LQR first-order optimality conditions
in~\cref{primal-dual-linear-system} has inertia
\begin{equation*}
\begin{aligned}
(N(n+m)+n, (N+1)n, 0)
\end{aligned}
\end{equation*}
if and only if
\begin{equation*}
\begin{aligned}
S_i \succ 0 \quad (i=0,\ldots,N),
\qquad
G_i \succ 0 \quad (i=0,\ldots,N-1).
\end{aligned}
\end{equation*}
\end{theorem}

\begin{proof}
Let $K_{\mathrm{LQR}}$ denote the left-hand side matrix in
~\cref{primal-dual-linear-system}. By~\cref{sylvester-inertia-lemma},
\begin{equation*}
\begin{aligned}
\operatorname{In}(K_{\mathrm{LQR}})
= \operatorname{In}(-\Delta)
  + \operatorname{In}(P+C^T\Delta^{-1}C).
\end{aligned}
\end{equation*}
Since $\Delta \succ 0$, the desired inertia of $K_{\mathrm{LQR}}$ is
equivalent to
\begin{equation*}
\begin{aligned}
P+C^T\Delta^{-1}C \succ 0.
\end{aligned}
\end{equation*}

The sequential Riccati recursion is a block $LDL^T$ factorization of the
reduced primal Hessian, using the stage order encoded by the dynamics. At each
backward step, eliminating $x_{i+1}$ uses the pivot
$P_{i+1}+\Delta_{i+1}^{-1}$, and eliminating $u_i$ uses the pivot
$G_i$. The initial-state elimination uses the remaining pivot
$P_0+\Delta_0^{-1}$. These state pivots are precisely the matrices $S_i$.
By repeated application of Sylvester's law of inertia,
the inertia of $P+C^T\Delta^{-1}C$ is the sum of the inertias of the $S_i$
and $G_i$ pivots.
Therefore $P+C^T\Delta^{-1}C$ is positive-definite if and only if every
$S_i$ and every $G_i$ is positive-definite.
\end{proof}

The parallel Riccati recursion also factorizes the $\widehat S_i$ and $G_i$ matrices after the reverse associative scan,
so this inertia certificate is also conveniently available in that setting.

Thus, after a Riccati factorization, the desired inertia can be certified by
checking Cholesky factorizations of the $G_i$ and of the symmetric state
pivots $S_i$, or equivalently of the inverse-free representatives
$\widehat S_i$. If any check fails, the current Hessian
regularization is insufficient for the line-search descent guarantee from~\cref{inertia-al-descent-theorem}.

\section{Residual Computation}

The residuals of the Newton-KKT linear systems from~\cref{primal-dual-linear-system}
(when specialized to discrete-time optimal control problems as in~\cref{optimal-control-specialization})
are given by
\begin{equation*}
\begin{aligned}
\begin{bmatrix}
\begin{bmatrix}
Q_i x_i + M_i u_i + A_i^T y_{i+1} + q_i - y_i \\
M_i^T x_i + R_i u_i + B_i^T y_{i+1} + r_i \\
\end{bmatrix}_{i=0, \ldots, N-1} \\
Q_N x_N + q_N - y_N \\
c_0 - \Delta_0 y_0 - x_0 \\
\begin{bmatrix}
A_i x_i + B_i u_i + c_{i+1} - \Delta_{i+1} y_{i+1} - x_{i+1}
\end{bmatrix}_{i=0, \ldots, N-1}
\end{bmatrix}
\end{aligned}
\end{equation*}

Thus, the residuals can be computed in $O(\log(m) + \log(n))$ parallel time.
Residual computation is only necessary when applying iterative refinement to solve
~\cref{primal-dual-linear-system} to higher precision.

\section{Cross-Stage Variables}\label{cross-stage-variables}

The Riccati recursions derived above rely on the fact that, after the standard
interior-point eliminations, the only constraints coupling consecutive stages are
the dynamics. In some applications, however, it is useful to introduce a small
number of variables that are shared by every stage. We call these
\emph{cross-stage variables} and denote them by $\theta \in \mathbb{R}^{p}$.
For example, in the quadpendulum benchmark, a scalar $\theta$ represents a
clearance margin that appears in the obstacle-avoidance constraints at every
time step.

We do not fold $\theta$ into the state. Doing so would require artificial
dynamics $\theta_{i+1}=\theta_i$ and would increase the state dimension of every
Riccati step by $p$, even when $p$ is small. Instead, we keep $\theta$ as a
global variable and exploit the resulting arrowhead structure.

Let $w$ collect all stagewise primal and dual Newton variables after the
stagewise slack, equality, and inequality eliminations described in
\cref{application}. Without $\theta$, the Newton system has the
dual-regularized LQR form
\begin{equation*}
K \Delta w = -r_w,
\end{equation*}
where multiplication by $K^{-1}$ corresponds exactly to a Riccati solve. With cross-stage variables,
the same linearization has the block form
\begin{equation}
\label{cross-stage-kkt}
\begin{bmatrix}
K & J_{\theta} \\
J_{\theta}^{T} & H_{\theta\theta}
\end{bmatrix}
\begin{bmatrix}
\Delta w \\
\Delta \theta
\end{bmatrix}
= -
\begin{bmatrix}
r_w \\
r_{\theta}
\end{bmatrix}.
\end{equation}
Here $J_{\theta}$ contains the derivatives with respect to $\theta$ of the
stage costs, dynamics, equalities, and inequalities after the same
interior-point eliminations have been applied, and $H_{\theta\theta}$ is the
second derivative of the Lagrangian with respect to $\theta$.

We solve~\cref{cross-stage-kkt} by forming the dense Schur complement in the
cross-stage variables:
\begin{equation}
\label{cross-stage-schur}
S_{\theta} = H_{\theta\theta} - J_{\theta}^{T} K^{-1} J_{\theta}.
\end{equation}
The required products with $K^{-1}$ are computed by reusing the Riccati
factorization: one solve is applied to $r_w$, and $p$ additional solves are
applied to the columns of $J_{\theta}$. Then
\begin{equation*}
\begin{aligned}
\Delta \theta
&= -S_{\theta}^{-1}
\left(r_{\theta} - J_{\theta}^{T} K^{-1} r_w\right), \\
\Delta w
&= -K^{-1} (r_w + J_{\theta}\Delta\theta).
\end{aligned}
\end{equation*}
Since $p$ is typically much smaller than the number of stagewise variables, the
extra dense factorization is negligible compared with the Riccati solves. In
particular, the sequential implementation requires $p+1$ Riccati solves plus an
$O(p^3)$ dense solve, while the parallel implementation can compute the $p$
column solves independently once the Riccati factorization is available.

This preserves the main advantage of the optimal-control-specific linear solve:
the large block $K$ remains an LQR-structured system, while the nonlocal
dependence introduced by $\theta$ is isolated in a small dense Schur complement.
It also keeps the modeling interface simple: stage cost, dynamics, equality, and
inequality callbacks may depend on $(x_i,u_i,\theta,i)$, with $\theta$ passed
unchanged to every stage.

\section{Software Contributions}\label{software}

We provide JAX implementations of both the sequential and parallel
dual-regularized LQR algorithms~\cite{rlqr-jax},
including unit tests that verify the correctness of the computed solutions
on random examples satisfying the required definiteness properties.

Furthermore, we released Primal-Dual LIPA~\cite{lipa}, a JAX implementation of
a regularized interior-point solver for discrete-time constrained optimal control
problems based on the methods described in this paper.

Existing modern interior-point solvers, such as IPOPT~\cite{ipopt}, support several generic
linear system solvers, but are not designed to accommodate linear solvers specialized
to specific problem types (e.g. optimal control problems). In fact, until recently, IPOPT
installations did not include all of the required headers for users to integrate any
custom user-side linear system solvers, which hindered the use of specialized
linear system solvers.

To address this limitation, we released a simple regularized interior-point solver
in~\cite{sip}, which allows users to register callbacks that specialize certain operations to their
problem type, including:
\begin{itemize}
\item a KKT system factorization callback;
\item a KKT system solve callback;
\item a KKT system residual computation callback.
\end{itemize}

This splits the solver into a shared backend and a problem-specific frontend.
We do exactly that for the case of optimal control in~\cite{sipoc}.

In addition, we implemented an integration with QDLDL~\cite{osqp} supporting arbitrary
user-provided KKT permutations (for optimal fill-in prevention) in~\cite{sip-qdldl},
resulting in an efficient sparse interior-point solver that avoids dense operations entirely.
Alternatively, a sparse linear algebra code generation library, such as SLACG~\cite{slacg},
can be used to solve the KKT systems and compute their residuals; this often achieves substantial
speedups for sparse problems, as no index computations are performed at runtime.

Usage examples can be found in~\cite{sipex}.
With correct usage, all of these implementations avoid dynamic memory allocation.

All of the provided packages are free and open-source (MIT-licensed).

\section{Benchmarks}\label{benchmarks}

We benchmarked the JAX implementation of Primal-Dual LIPA~\cite{lipa}
and the Python bindings for SIP~\cite{sip-python} against established
optimal-control and nonlinear-programming solvers:
acados~\cite{acados}, aligator~\cite{aligator}, CSQP~\cite{csqp},
Fatrop~\cite{fatrop}, IPOPT~\cite{ipopt}, and Trajax~\cite{trajax}.
All solvers were run on the same multiple-shooting formulation of each test problem,
with optimal problem-dependent settings. The benchmark harness records
only the timed solve call: JAX compilation, CasADi graph construction,
code generation, and one-time solver setup are excluded.

\begin{figure*}[t]
\centering
\begin{minipage}[t]{0.49\textwidth}
\centering
\includegraphics[width=\linewidth]{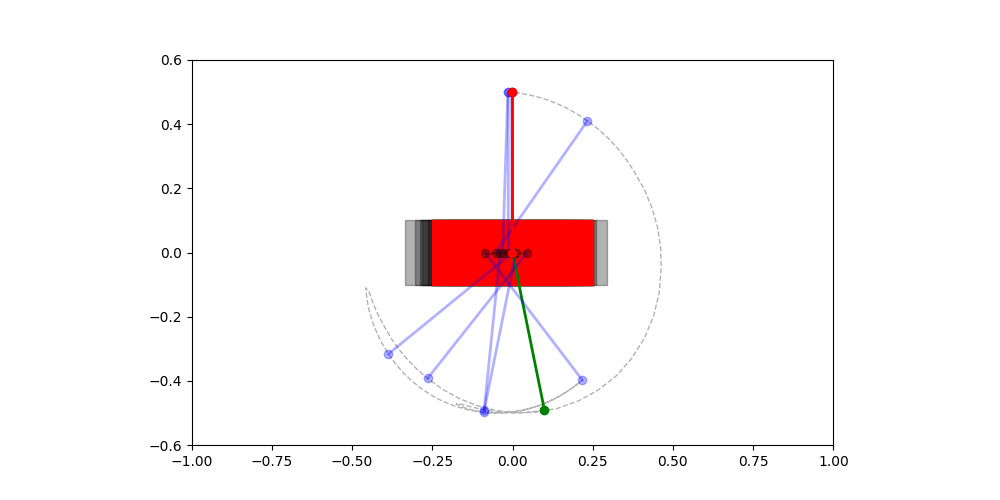}
\end{minipage}
\hfill
\begin{minipage}[t]{0.49\textwidth}
\centering
\includegraphics[width=\linewidth]{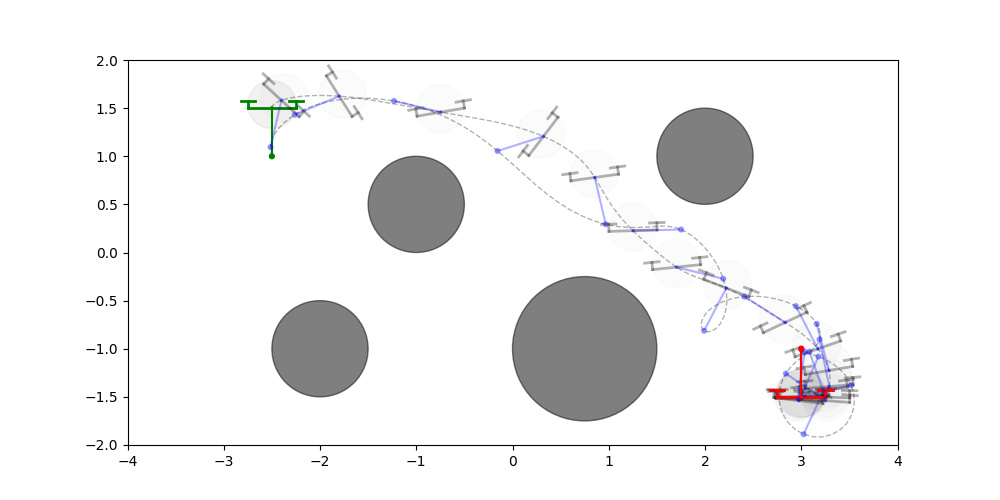}
\end{minipage}
\caption{Representative optimized trajectories from the analytical benchmark
set. Left: bounded-actuation cartpole swing-up. Right: quadpendulum navigation
with obstacle-avoidance constraints. Faded configurations show intermediate
states along each optimized trajectory.}
\label{fig:benchmark-trajectories}
\end{figure*}

\begin{figure*}[t]
\centering
\includegraphics[width=\textwidth]{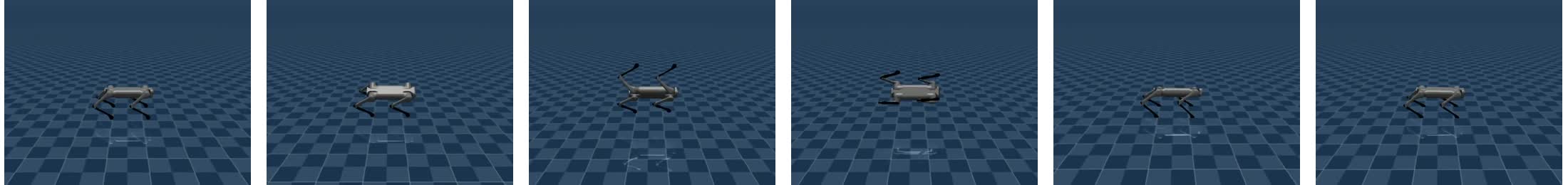}\\[0.4em]
\includegraphics[width=\textwidth]{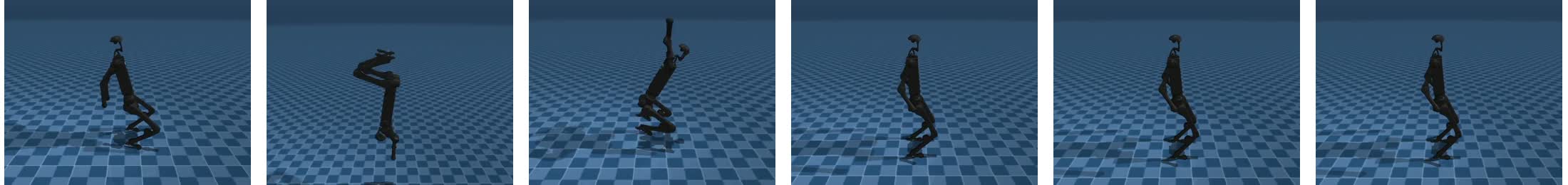}
\caption{Uniformly sampled frames from optimized MJX whole-body trajectories.
Top: quadruped barrel roll. Bottom: humanoid backflip.}
\label{fig:mjx-trajectory-frames}
\end{figure*}

\begin{table}[H]
\caption{Benchmark problems. Here $T$ is the horizon length, $n$ and $m$
are the state and control dimensions, $n_\theta$ is the number of
cross-stage variables, and $n_g$ is the number of stagewise inequality
constraints.}
\label{tab:benchmark-problems}
\centering
\scriptsize
\begin{tabular}{lrrrrl}
\toprule
Problem & $T$ & $n$ & $m$ & $n_\theta$ & Constraints \\
\midrule
Cartpole & 50 & 4 & 1 & 0 & terminal goal, $n_g=2$ \\
Acrobot & 50 & 4 & 1 & 0 & dynamics only \\
Quadpendulum & 200 & 8 & 2 & 0 & terminal goal, $n_g=48$ \\
Quadpendulum-$\theta$ & 200 & 8 & 2 & 1 & terminal goal, $n_g=48$ \\
Barrel roll & 100 & 61 & 12 & 0 & MJX dynamics, $n_g=52$ \\
Backflip & 130 & 63 & 31 & 0 & MJX dynamics, $n_g=54$ \\
Double jump & 100 & 63 & 31 & 0 & MJX dynamics, $n_g=54$ \\
Trotting & 92 & 61 & 12 & 0 & MJX dynamics, $n_g=52$ \\
\bottomrule
\end{tabular}
\end{table}

\begin{table}[H]
\caption{Analytical-problem iteration counts. Entries are
``iterations status'', where ``ok'' denotes convergence, ``x'' denotes
that the solver ran but did not satisfy the benchmark success criterion,
``timeout'' denotes a subprocess hard timeout, and ``--'' denotes an
unsupported problem class.}
\label{tab:benchmark-analytical-iters}
\centering
\scriptsize
\begin{tabular}{lrrrr}
\toprule
Solver & Cartpole & Acrobot & Quadpendulum & Quadpendulum-$\theta$ \\
\midrule
LIPA-GPU & 84 ok & 108 ok & 81 ok & 143 ok \\
LIPA-CPU & 82 ok & 108 ok & 81 ok & 143 ok \\
SIP-CasADi & 231 ok & 68 ok & 231 ok & 198 ok \\
SIP-JAX & 231 ok & 68 ok & 123 ok & 198 ok \\
IPOPT-CasADi & 33 ok & 21 ok & 65 ok & 103 ok \\
IPOPT-JAX & 46 ok & 21 ok & 163 ok & 82 ok \\
acados & 68 ok & 99 ok & 31 ok & -- \\
Fatrop-CasADi & 80 ok & 17 ok & 112 ok & -- \\
Fatrop-JAX & 99 ok & 17 ok & 112 ok & -- \\
CSQP-CasADi & 115 ok & 74 ok & 1000 x & -- \\
CSQP-JAX & 113 ok & 74 ok & 1000 x & -- \\
aligator-CasADi & 101 ok & 55 ok & 1000 x & -- \\
aligator-JAX & 114 ok & 55 ok & timeout & -- \\
Trajax & 310 ok & 86 ok & 237 ok & -- \\
\bottomrule
\end{tabular}
\end{table}

\begin{table}[H]
\caption{Analytical-problem benchmark results. Entries are solve times in
seconds; ``x'' indicates that the solver ran but did not satisfy the
benchmark success criterion, and ``--'' indicates that the adapter does
not support the problem class.}
\label{tab:benchmark-analytical}
\centering
\scriptsize
\begin{tabular}{lrrrr}
\toprule
Solver & Cartpole & Acrobot & Quadpendulum & Quadpendulum-$\theta$ \\
\midrule
LIPA-GPU & 0.134 & 0.136 & 0.171 & 0.351 \\
LIPA-CPU & 0.092 & 0.051 & 0.668 & 1.358 \\
SIP-CasADi & 1.051 & 0.298 & 6.328 & 5.725 \\
SIP-JAX & 2.379 & 0.396 & 4.164 & 7.457 \\
IPOPT-CasADi & 0.058 & 0.034 & 4.000 & 6.769 \\
IPOPT-JAX & 5.638 & 1.924 & 95.391 & 51.383 \\
acados & 0.057 & 0.012 & 0.527 & -- \\
Fatrop-CasADi & 0.041 & 0.030 & 1.596 & -- \\
Fatrop-JAX & 37.155 & 4.265 & 189.991 & -- \\
CSQP-CasADi & 3.929 & 0.570 & 315.738x & -- \\
CSQP-JAX & 22.325 & 7.517 & 1077.415x & -- \\
aligator-CasADi & 1.299 & 0.515 & 154.943x & -- \\
aligator-JAX & 16.404 & 4.989 & 360.529x & -- \\
Trajax & 5.396 & 0.876 & 12.248 & -- \\
\bottomrule
\end{tabular}
\end{table}

\Cref{tab:benchmark-problems} summarizes the test set. The first four
problems have analytical dynamics and CasADi mirrors, allowing both
symbolic and JAX-backed adapters to be tested. Cartpole is a swing-up
with bounded actuation; acrobot is an unconstrained swing-up;
quadpendulum is a planar quadrotor-pendulum navigation problem with
world-boundary, obstacle-avoidance, attitude, and thrust constraints; and
quadpendulum-$\theta$ adds a global clearance-margin variable shared
across stages. The analytical examples are modified from earlier
examples in~\cite{trajax}. The last four problems use MuJoCo MJX dynamics for
whole-body robotics: a quadruped barrel roll, a humanoid backflip, a
humanoid double jump, and a quadruped trot. These problems include
contact, friction, torque, and joint-velocity constraints, and are the
main test of the JAX/GPU implementation; they are modified from the
examples in~\cite{pd-ilqr-legged}. Representative MJX optimized
trajectories are shown in~\cref{fig:mjx-trajectory-frames}.

A solve is marked successful when the adapter reports convergence and
the returned trajectory satisfies the benchmark tolerance: $10^{-6}$ on
the analytical problems and $10^{-3}$ on the MJX problems. Equality
violation is the infinity norm of the initial-state defect, dynamics
defects, and user equalities; inequality violation is
$\|\max(0,g)\|_\infty$. The benchmark harness also recomputes a common
KKT residual decomposition when an adapter exposes compatible
multipliers, but we report wall-clock time and success status here
because multiplier availability and sign conventions differ across
packages.

\begin{table}[H]
\centering
\caption{MJX whole-body iteration counts. Entries follow the convention
of~\cref{tab:benchmark-analytical-iters}.}
\label{tab:benchmark-mjx-iters}
\scriptsize
\begin{tabular}{lrrrr}
\toprule
Solver & Barrel & Backflip & Jump & Trot \\
\midrule
LIPA-GPU & 70 ok & 476 ok & 223 ok & 29 ok \\
SIP-JAX & 186 ok & timeout & timeout & 55 ok \\
IPOPT-JAX & 388 x & 211 x & 285 x & 117 ok \\
Fatrop-JAX & timeout & timeout & timeout & 124 ok \\
CSQP-JAX & 200 x & 200 x & 400 x & 212 x \\
Trajax & timeout & 196 ok & 1027 ok & timeout \\
\bottomrule
\end{tabular}
\end{table}

\begin{table}[H]
\centering
\caption{MJX whole-body benchmark results. Entries are solve times in
seconds from the merged report; ``x'' marks failures or timeouts.}
\label{tab:benchmark-mjx}
\scriptsize
\begin{tabular}{lrrrr}
\toprule
Solver & Barrel & Backflip & Jump & Trot \\
\midrule
LIPA-GPU & 3.25 & 31.69 & 12.96 & 1.47 \\
SIP-JAX & 102.20 & 2400.73x & 2400.60x & 25.62 \\
IPOPT-JAX & 1201.63x & 1202.47x & 1203.48x & 324.95 \\
Fatrop-JAX & 2400.40x & 2400.53x & 2400.40x & 1494.62 \\
CSQP-JAX & 352.01x & 656.88x & 948.53x & 312.34x \\
Trajax & 2400.51x & 240.76 & 1043.63 & 2400.60x \\
\bottomrule
\end{tabular}
\end{table}

\Cref{tab:benchmark-analytical-iters,tab:benchmark-analytical} show that,
on the smallest analytical problems, specialized CPU solvers with symbolic
models can be faster than LIPA. On the larger constrained quadpendulum
problem, however, LIPA-GPU solves the instance in $0.171$ seconds, compared to
$1.596$ seconds for Fatrop-CasADi, $4.000$ seconds for IPOPT-CasADi,
and $4.164$ seconds for SIP-JAX. LIPA is also one of the few solvers in
the benchmark that supports the cross-stage variable in
quadpendulum-$\theta$.

\Cref{tab:benchmark-mjx-iters,tab:benchmark-mjx} show a larger separation on the MJX
whole-body tasks. LIPA-GPU is the only solver in the comparison that
successfully solves all four MJX problems. The times reported in the merged
report are $3.25$ seconds for the quadruped barrel roll, $31.69$ seconds
for the humanoid backflip, $12.96$ seconds for the humanoid double
jump, and $1.47$ seconds for trotting. Across five independent LIPA-GPU
MJX repeats, the corresponding median solve times were $3.20$,
$29.09$, $12.96$, and $1.36$ seconds. The remaining solvers either
timed out on at least two of the four MJX tasks or converged only on
the easier instances.

Overall, these results support two conclusions.
First, pairing the regularized primal-dual interior-point method
with a dual-regularized LQR Newton-KKT solver
delivers robust performance on complex trajectory optimization problems,
even in the presence of stagewise equality and inequality constraints.
Second, the GPU-accelerated version of the solver enables substantial timing gains
as the horizon length, problem dimension, and constraint count increase.

\section{Conclusion}\label{conclusion}

We showed that the regularized interior-point method provides a practical foundation
for fast constrained discrete-time optimal control via multiple-shooting.
Its advantages include the following:
\begin{itemize}
\item it poses Newton-KKT systems that can be reduced to dual-regularized LQR problems,
  and thus solved highly efficiently, both sequentially and in parallel;
\item it is extremely simple to implement;
\item it provides robust handling of equality and inequality constraints;
\item it imposes no rank requirements on constraint Jacobians;
\item it efficiently supports cross-stage variables;
\item it is guaranteed to descend the augmented barrier-Lagrangian merit function;
\item it does not suffer from the Maratos effect, and thus avoids second-order corrections.
\end{itemize}

Our open-source implementations and benchmarks show that this method achieves competitive performance
across a wide set of complex trajectory optimization problems, with a clear advantage on larger problems.

\bibliographystyle{siamplain}
\bibliography{references}

\end{document}